\documentclass[11pt,reqno]{amsart}
\usepackage{amsmath,amscd,amssymb,amsfonts,amsthm,courier,relsize,bm}
\usepackage{hyperref,enumitem,mathrsfs,mathtools,slashed}

\usepackage{xcolor}  
\hypersetup{
    colorlinks,
    linkcolor={red!50!black},
    citecolor={blue!70!black},
    urlcolor={blue!80!black}
}

\newtheorem{theorem}{Theorem}[section]
\newtheorem{corollary}[theorem]{Corollary}
\newtheorem{proposition}[theorem]{Proposition}

\newtheorem{lemma}[theorem]{Lemma}
\theoremstyle{definition}    
\newtheorem{definition}[theorem]{Definition}
\theoremstyle{remark}

\newtheorem{remark}[theorem]{Remark}

\newtheorem{example}[theorem]{Example}

\newcommand{\pair}[2]{\langle #1, #2 \rangle}
\newcommand{\ignore}[1]{}
\newcommand{\matr}[4]{\left(\begin{array}{cc}#1&#2\\#3&#4\end{array}\right)}
\newcommand{\ol}[1]{\overline{#1}}
\newcommand{\ul}[1]{\underline{#1}}
\newcommand{\sul}[1]{\underline{#1\mkern-4mu}\mkern4mu }
\newcommand{\ti}[1]{\widetilde{#1}}

\newcommand{\mf}[1]{\mathfrak{#1}}

\newcommand{\tn}[1]{\textnormal{#1}}

\def\dirac{\ensuremath{\slashed{\partial}}}
\def\d{\ensuremath{\mathrm{d}}}

\def\Ad{\ensuremath{\textnormal{Ad}}}
\def\ad{\ensuremath{\textnormal{ad}}}
\def\g{\ensuremath{\mathfrak{g}}}
\def\k{\ensuremath{\mathfrak{k}}}
\def\t{\ensuremath{\mathfrak{t}}}

\def\h{\ensuremath{\mathfrak{h}}}

\def\z{\ensuremath{\mathfrak{z}}}

\def\B{\ensuremath{\mathcal{B}}}
\def\C{\ensuremath{\mathcal{C}}}
\def\D{\ensuremath{\mathcal{D}}}

\def\I{\ensuremath{\mathcal{I}}}

\def\L{\ensuremath{\mathcal{L}}}

\def\N{\ensuremath{\mathcal{N}}}
\def\O{\ensuremath{\mathcal{O}}}

\def\Q{\ensuremath{\mathcal{Q}}}

\def\T{\ensuremath{\mathcal{T}}}

\def\Z{\ensuremath{\mathcal{Z}}}

\def\bC{\ensuremath{\mathbb{C}}}
\def\bR{\ensuremath{\mathbb{R}}}
\def\bZ{\ensuremath{\mathbb{Z}}}

\def\End{\ensuremath{\textnormal{End}}}
\def\Hom{\ensuremath{\textnormal{Hom}}}

\def\ker{\ensuremath{\textnormal{ker}}}
\def\supp{\ensuremath{\textnormal{supp}}}

\def\Tr{\ensuremath{\textnormal{Tr}}}
\def\pr{\ensuremath{\textnormal{pr}}}

\def\pt{\ensuremath{\textnormal{pt}}}

\def\Cl{\ensuremath{\textnormal{Cl}}}

\def\index{\ensuremath{\textnormal{index}}}

\def\Ch{\ensuremath{\textnormal{Ch}}}
\def\Td{\ensuremath{\textnormal{Td}}}

\def\Ahat{\ensuremath{\widehat{\textnormal{A}}}}

\title{Log symplectic manifolds and $[Q,R]=0$}
\author{Yi Lin}\address{Georgia Southern University}\email{yilin@georgiasouthern.edu}
\author{Yiannis Loizides}\address{Cornell University}\email{yl3542@cornell.edu (corresponding author)}
\author{Reyer Sjamaar}\address{Cornell University}
\email{sjamaar@math.cornell.edu}
\author{Yanli Song}\address{Washington University in St. Louis}\email{yanlisong@wustl.edu}

\begin{document}
\sloppy
\maketitle

\vspace{-0.7cm}

\begin{abstract}
We show, under an orientation hypothesis, that a log symplectic manifold with simple normal crossing singularities has a stable almost complex structure, and hence is Spin$_c$. In the compact Hamiltonian case we prove that the index of the Spin$_c$ Dirac operator twisted by a prequantum line bundle satisfies a $[Q,R]=0$ theorem.
\end{abstract}

\section{Introduction}
A $b$ symplectic manifold is a manifold $M^n$ ($n$ even) equipped with a Poisson bivector $\pi$ such that $\pi^{n/2}$ vanishes transversely along a hypersurface $Z \subset M$.  The inverse $\pi^{-1}=\omega$ may be thought of informally as a symplectic form with a pole along $Z$, or more precisely as a smooth section of $\wedge^2({}^bT^*M)$, where $^bTM$ is the $b$ tangent bundle in the sense of Melrose.  $b$ symplectic manifolds were introduced by Nest and Tsygan \cite{NestTsygan96} in the context of deformation quantization.  They were studied extensively in \cite{guillemin2014symplectic,guillemin2014convexity, gualtieri2014symplectic}, where it was found that many well-known results from symplectic geometry have analogues in this setting.

In \cite{guillemin2018geometric} Guillemin, Miranda and Weitsman described a formal geometric quantization for compact prequantized $b$-symplectic manifold endowed with a Hamiltonian action of a torus $T$ assumed to have non-zero modular weights (see Definition \ref{d:modwt}); under this assumption all of the reduced spaces are actually symplectic, and the formal quantization was defined using the Guillemin-Sternberg \cite{GuilleminSternbergConjecture} quantization-commutes-with-reduction ($[Q,R]=0$) principle. The authors proved the surprising result that the formal quantization thus defined is finite dimensional, and posed the problem of finding a Fredholm operator whose equivariant index equals the quantization. Two possible answers to this question were offered in \cite{BLSbsympl}. The first approach involved an Atiyah-Patodi-Singer-type index on a manifold with boundary obtained by removing a small neighborhood of the hypersurface. The second approach involved constructing a Spin$_c$ structure on the whole manifold $M$, and taking the index of the Spin$_c$ Dirac operator.  It was proved that these two approaches agree and satisfy the $[Q,R]=0$ principle, hence also agree with the formal quantization of Guillemin, Miranda and Weitsman.

In this article we revisit and extend the second approach of \cite{BLSbsympl} mentioned above.  We work in the more general setting, introduced and studied in \cite{gualtieri2017tropical}, in which the symplectic form is permitted to have poles along a simple normal crossings divisor.  In this article we will refer to such a singular symplectic form as a `log symplectic form', in order to avoid confusion with the more restricted $b$ symplectic setting where the divisor is required to be a smooth hypersurface. (Another suitable term would be `$c$ symplectic form', where `$c$' is for `corner'.)

We observe that a log symplectic manifold with a simple normal crossing divisor that admits global defining functions (Definition \ref{d:IZ}), possesses stable almost complex structures (Corollary \ref{c:spinc}).  Since stably almost complex manifolds are Spin$_c$, we may use the Spin$_c$ structure (and a prequantum line bundle) to define the `quantization' of a compact log symplectic manifold in terms of the index of a Spin$_c$ Dirac operator.  Given a proper momentum map for the action, the regular reduced spaces are also compact log symplectic manifolds \cite{LLSS1}.  This leads us to formulate and prove a $[Q,R]=0$ theorem in this context, extending the results of \cite{BLSbsympl}. We explain how these results specialize to the toric log symplectic manifolds considered in \cite{gualtieri2017tropical}.

\bigskip

\noindent \textbf{Acknowledgements.} We thank Eva Miranda, Eckhard Meinrenken, Marco Gualtieri, Peter Crooks and Ralph Klaase for helpful conversations. We thank the referees for their careful reading of the manuscript and helpful suggestions. Y. Song is supported by NSF grants DMS-1800667, 1952557.

\bigskip

\noindent \textbf{Notation and conventions.} If $V$ is a vector space, then $\sul{V}$ will denote the trivial vector bundle with fibre $V$ over the base (understood from context). If a Lie group $G$ with Lie algebra $\g$ acts on a manifold, the vector field generated by an element $X \in \g$ is denoted $X_M$, and its value at $m \in M$ is the derivative at $t=0$ of $\exp(-tX)\cdot m$.

\section{Normal crossing divisors}
\subsection{Normal crossing divisors and the log tangent bundle.}
Let $M^n$ be a connected manifold without boundary. 

\begin{definition}
By a \emph{(simple) normal crossing divisor} $(M,\Z)$ we mean a finite collection $\Z$ of embedded real codimension $1$ connected hypersurfaces in $M$ such that if $Z_1,...,Z_k \in \Z$ and $p \in Z_1 \cap \cdots \cap Z_k$, there is a local coordinate chart $\varphi\colon U \rightarrow \bR^n$ centred at $p$ that maps $Z_j\cap U$ into a subset of the coordinate hyperplane $x_j=0$ in $\bR^n$.  We refer to such a chart $(U,\varphi)$ as a \emph{normal crossing chart}.  In the sequel we will omit the adjective `simple'.
\end{definition}
The definition implies that any $n+1$-fold intersection of the hypersurfaces is empty. It is well-known that if $\Z$ is a normal crossing divisor, then the sheaf of smooth vector fields tangent to all $Z \in \Z$ is locally free and finitely generated, hence corresponds to smooth local sections of a vector bundle $T_\Z M\rightarrow M$.

\begin{definition}
Let $(M^n,\Z)$ be a normal crossing divisor.  The \emph{log tangent bundle} $T_\Z M$ is the smooth rank $n$ vector bundle over $M$ whose sheaf of smooth sections consists of local vector fields tangent to all $Z \in \Z$.  The log tangent bundle is a Lie algebroid, with anchor and bracket determined by the inclusion $\Gamma(T_\Z M)\subset \Gamma(TM)$. 
\end{definition}

Sections of the dual vector bundle $T^*_\Z M$ may be thought of as singular $1$-forms with at worst simple poles along $\Z$.  For example, if $f$ is a smooth function vanishing to order $1$ along one of the hypersurfaces $Z \in \Z$, and non-vanishing on $M\backslash Z$, the singular 1-form $\d f/f$ may be thought of as a smooth section of the vector bundle $T_\Z^*M$, since there is an obvious way to make sense of its pairing with any smooth vector field tangent to $Z$.

\begin{definition}
Sections of the exterior algebra bundle $\Gamma(\wedge^\bullet T_\Z^*M)=\Omega^\bullet(M,\Z)$ will be referred to as \emph{log differential forms}.  As a special case of Lie algebroid cohomology, $\Omega^\bullet(M,\Z)$ carries a de Rham differential and the corresponding cohomology groups are called the \emph{log de Rham cohomology} $H^\bullet(M,\Z)$.
\end{definition}

\begin{definition}
\label{d:res}
For $Z \in \Z$ there is a \emph{residue} map
\[ \tn{res}_Z \colon \Omega^\bullet(M,\Z)\rightarrow \Omega^{\bullet-1}(Z,\Z\cap Z) \]
where $\Z \cap Z=\{Z \cap W|W \in \Z, W\ne Z\}$ is the induced divisor on the hypersurface $Z$.  One definition of $\tn{res}_Z$ is in terms of the contraction
\[ \tn{res}_Z(\alpha)=\iota(e_Z)\alpha|_Z \]
where $e_Z \in \Gamma(T_\Z M|_Z)$ is the canonical non-vanishing section, given in any normal crossing chart $(U,\varphi)$ mapping $U\cap Z$ into the coordinate hyperplane $x_1=0$, by restriction to $U \cap Z$ of the local section $x_1 \partial/\partial x_1$ of $T_\Z M$.
\end{definition}
Collectively the partially defined sections $e_Z$ for $Z \in \Z$ determine a canonical basis for the kernel of the anchor map $T_\Z M\rightarrow TM$ at every point in $M$. 

Taking residues is compatible with de Rham differentials. For a log differential form $\alpha$, the residue $\tn{res}_Z(\alpha)=0$ if an only if $\alpha$ may be regarded as an element of $\Omega(M,\Z\backslash \{Z\})$.  In particular a log differential form may be regarded as an ordinary smooth form if and only if all of its residues vanish.  These considerations lead to a version of the Mazzeo-Melrose theorem (cf. \cite{melrose1993atiyah, guillemin2014symplectic} for the case without crossings, \cite[Section A.24]{gualtieri2017tropical} for the case with crossings): the logarithmic de Rham cohomology
\begin{equation} 
\label{e:MazzeoMelrose}
H^p(M,\Z)\simeq H^p(M)\oplus \prod_i H^{p-1}(Z_i)\oplus \prod_{i<j} H^{p-2}(Z_i\cap Z_j) \oplus \cdots 
\end{equation}
with the groups on the RHS being ordinary de Rham cohomology.

\subsection{Orientations and stable isomorphism.}
\begin{definition}
\label{d:IZ}
Let $Z \subset M$ be an embedded hypersurface. The sheaf of smooth real-valued functions vanishing on $Z$ is the sheaf of smooth sections of a real line bundle $\I_Z$. Note that $\I_Z$ is trivial if and only if $Z$ \emph{admits a global defining function}: a smooth real-valued function $f$ such that $f^{-1}(0)=Z$ and $\d f|_Z$ is non-vanishing. If $\Z$ is a normal crossing divisor, then we will say that $\Z$ \emph{admits global defining functions} if each $Z \in \Z$ admits a global defining function.
\end{definition}

For a normal crossing divisor $\Z$, let
\[ \I_\Z=\bigotimes_{Z \in \Z} \I_Z.\]
The line bundles $\I_\Z$, $\det(TM)$ and $\det(T_\Z M)$ are related by
\begin{equation}
\label{e:detoz}
\det(TM)\otimes \I_\Z \simeq \det(T_\Z M), 
\end{equation}
via the map on sections $\nu \otimes f \mapsto f\nu$. In general the vector bundles $TM$, $T_\Z M$ are not isomorphic.  
\begin{example}
\label{ex:circle}
If $M=S^1$ and $\Z$ is a single point, then $T_\Z M$ is the non-trivial real line bundle over $S^1$.
\end{example}
\begin{example}[2-sphere]
\label{ex:2sphere0}
Let $M=S^2$ be the unit sphere in $\bR^3$ centred at the origin $x=y=z=0$.  Let $0\le \theta \le 2\pi$, $0\le \varphi \le \pi$ be spherical coordinates, where $\theta$ is the angle in the $x,y$ plane, and $\varphi$ is the angle to the positive $z$ axis.  Let $\Z=\{Z\}$ be the hypersurface $z=0$. The vector bundle $T_\Z M$ is trivial.  Indeed it suffices to produce one non-vanishing global section: for example, take the vector field $\sin(\varphi)\partial_\varphi$ which vanishes only at $(0,0,\pm 1)$, and perform a rigid rotation by $\pi/2$ about the $x$-axis, so that the zeros sit at the points $(0,\pm 1,0) \in Z$; the resulting vector field represents a global non-vanishing section of $T_\Z M$.  To obtain examples with crossings, we can add additional hypersurfaces, for example $\Z=\{\{x=0\},\{y=0\},\{z=0\}\}$ is a normal crossing divisor with $T_\Z M$ trivial. 
\end{example}  

\begin{theorem}
\label{t:stableiso}
Let $M$ be a manifold and let $\Z$ be a normal crossing divisor that admits global defining functions. Then there exists an isomorphism
\begin{equation}
\label{e:stableiso1}
\ul{\bR}\oplus T_\Z M\simeq \ul{\bR}\oplus TM.
\end{equation}
After choosing orientations on the line bundles $\I_Z$, the construction is canonical up to homotopy.
\end{theorem}
\begin{remark}
It was pointed out to us that at least two related results already appear in the literature. In the case of a hypersurface without crossing, Theorem \ref{t:stableiso} was stated in \cite[p.43, Remark]{guillemin2000unfolding} and proved in \cite[Proposition 2.3]{klaasesymplie2020}.
\end{remark}
\begin{proof}
We proceed by induction on the number of hypersurfaces in $\Z$. Let $\Z'$, $\Z=\Z'\cup \{Z\}$ be normal crossing divisors where $\Z$ contains one additional hypersurface $Z$. By induction we may assume we have already constructed an isomorphism $\ul{\bR}\oplus T_{\Z'}M \rightarrow \ul{\bR}\oplus TM$. It therefore suffices to construct a further isomorphism $\ul{\bR} \oplus T_\Z M \rightarrow \ul{\bR}\oplus T_{\Z'}M$.

There is a canonical map
\[ \iota \colon T_\Z M \rightarrow T_{\Z'}M, \]
given by the obvious inclusion at the level of sheaves. We may find a global defining function $f$ for $Z$, and a vector field $V$ tangent to $\Z'$, such that $Vf=1$ holds on a neighborhood of $Z$; let $\rho \in C^\infty(M)$ be a bump function with support contained in this neighborhood and equal to $1$ on $Z$. We have a block diagonal bundle morphism
\begin{equation} 
\label{e:noninvert}
\matr{f}{0}{0}{\iota}\colon \ul{\bR}\oplus T_\Z M \rightarrow \ul{\bR}\oplus T_{\Z'}M, 
\end{equation}
which is an isomorphism away from $Z$. The strategy is to perturb the off-diagonal entries near $Z$ to obtain an isomorphism. Perturb \eqref{e:noninvert} to
\begin{equation} 
\label{e:invert}
\matr{f}{\rho \tfrac{\d f}{f}}{-\rho V}{\iota}\colon \ul{\bR}\oplus T_\Z M \rightarrow \ul{\bR} \oplus T_{\Z'}M. 
\end{equation}
In \eqref{e:invert}, $-\rho V$ is regarded as a bundle map $\ul{\bR}\rightarrow T_{\Z'}M$ that sends $1 \in \Gamma(\ul{\bR})$ to the vector field $-\rho V \in \Gamma(T_{\Z'}M)$.

Clearly \eqref{e:invert} is an isomorphism on $M \backslash \supp(\rho)$, so it suffices to consider points near $Z$ in the support of $\rho$. Modifying $\rho$ if necessary, we may assume $\supp(\rho)$ is covered by normal crossing charts centered at points $z \in Z$. Let $z \in Z$ and let $(U,(x_1=f|_U,...,x_k,x_{k+1},...,x_n))$ be a normal crossing chart centered at $z$, where $\Z\upharpoonright U$ consists of the hypersurfaces $x_1=0,...,x_k=0$. Without loss of generality we may also arrange that $V|_U=\partial/\partial x_1$. Local generators for $T_\Z M$ on $U$ are
\[ x_1\frac{\partial}{\partial x_1},...,x_k\frac{\partial}{\partial x_k},\frac{\partial}{\partial x_{k+1}},...,\frac{\partial}{\partial x_n}.\]
Local generators for $T_{\Z'}M$ are the same, except with $x_1\partial/\partial x_1$ replaced with $\partial/\partial x_1$. With respect to these local frames, the matrix representation of \eqref{e:invert} is
\[ \left(\begin{array}{ccc}x_1&\rho & 0\\-\rho &x_1 & 0\\ 0&0&1_{n-1}\end{array}\right).\]
The determinant $x_1^2+\rho^2$ does not vanish on $\supp(\rho)$. 

Up to homotopy, the construction described above only depends on homotopy classes of global defining functions for the hypersurfaces in $\Z$, or equivalently, on choices of orientations on the line bundles $\I_Z$, $Z \in \Z$.
\end{proof}
\begin{remark}
Assuming $M$ is connected, one convenient way to fix choices of orientations of the line bundles $\I_Z$ is to select a connected component of $M\backslash \cup \Z$, and then choose global defining functions that are $>0$ on that component.
\end{remark}

\begin{remark}
\label{r:disconnected}
To simplify notation in later sections, we have assumed each $Z \in \Z$ is connected, although this is not necessary in Theorem \ref{t:stableiso}. A similar construction works if $\I_\Z$ is trivial and if there is a defining function for each $Z \in \Z$ defined only on a neighborhood of $\cup \Z$.
\end{remark}

\begin{remark}
\label{r:fibre}
Suppose $\pi \colon M\rightarrow X$ is a fibre bundle and $\Z=\pi^{-1}(\Z_X)$ is the inverse image of a normal crossing divisor $\Z_X$ on $X$ that admits global defining functions.  Choosing a connection we obtain splittings
\[ TM\simeq \pi^*TX\oplus \ker(T\pi), \qquad T_\Z M\simeq \pi^*T_{\Z_X}X\oplus \ker(T\pi),\] 
and it is clear from the proof of Theorem \ref{t:stableiso} that the stable isomorphism $\ul{\bR}\oplus TM\simeq \ul{\bR}\oplus T_\Z M$ can be chosen compatible with these splittings.
\end{remark}

\begin{corollary}
\label{c:finitecover}
Any normal crossing divisor $(M,\Z)$ has a finite cover $(\hat{M},\hat{\Z})$ such that $T\hat{M}$, $T_{\hat{\Z}}\hat{M}$ are stably isomorphic.
\end{corollary}
\begin{proof}
For any real line bundle $E$, the bundle of fibre-orientations of $E$ is a 2-fold cover such that the pullback of $E$ is trivial. Thus for a suitable $2^{\# \Z}$-fold cover $\hat{M}$, the pullback of each of the line bundles $\I_Z$ becomes trivial.
\end{proof}

Since rational Pontryagin classes are insensitive to real line bundles and finite covers, we recover the following (see \cite[Appendix]{BLSbsympl} for the case without crossings, where this was explained in terms of Chern-Weil representatives).
\begin{corollary}
\label{c:Pont}
For any $(M,\Z)$, $TM$, $T_\Z M$ have the same rational Pontryagin classes.
\end{corollary}
See also \cite{klaasesymplie2020} for comparison of characteristic classes for $TM$, $T_\Z M$.

\section{Log symplectic manifolds}
\begin{definition}
Let $(M,\Z)$ be a normal crossing divisor.  A \emph{log symplectic form} is a closed log $2$-form $\omega \in \Omega^2(M,\Z)$ such that the map $T_\Z M \rightarrow T_\Z^*M$ induced by contraction with $\omega$ is an isomorphism.
\end{definition}

\begin{example}[2-sphere]
\label{ex:2sphere1}
Let $a \in (-1,1)$.  $M=S^2$ admits the log symplectic form
\[ \omega=\frac{\d z}{2\pi (z-a)}\d \theta.\]
for the divisor $\Z=\{\{z=a\}\}$. One can easily produce examples with crossings, for example, the log 2-form $\frac{\d z}{2\pi xyz}\d \theta$ with poles along the divisor $\{\{x=0\},\{y=0\},\{z=0\}\}$.
\end{example}

\subsection{Momentum maps.}
Let $G$ be a compact connected Lie group with Lie algebra $\g$, and let $\z\subset \g$ denote the centre.  Suppose that $G$ acts smoothly on $M$ preserving the log symplectic 2-form $\omega$, and such that each hypersurface $Z \in \Z$ is mapped to itself.  For $X \in \g$, let $X_M$ (resp. $X_Z$) denote the corresponding vector field on $M$ (resp. $Z$, for $Z \in \Z$).

\begin{definition}
\label{d:modwt}
The \emph{modular weight} of $Z \in \Z$ is the map $c_Z \colon Z \rightarrow \g^\ast$ defined by
\[ \pair{c_Z}{X}=\iota(X_Z)\tn{res}_Z(\omega).\]
\end{definition}
\begin{proposition}
$c_Z$ is constant and $\z^*$-valued.
\end{proposition}
\begin{proof}
Recall that $Z$ is connected.  To see that $c_Z$ is constant, use the Cartan homotopy formula and the compatibility of $\tn{res}_Z$ with de Rham differentials. Then the fact that $c_Z$ is $\z^*$-valued follows because it is a $G$-equivariant map $Z\rightarrow \g^*$.
\end{proof}

It is convenient to work with a differential form representative for the image of $\omega$ in $H^2(M)$ under the Mazzeo-Melrose map \eqref{e:MazzeoMelrose}. To this end we first describe the inclusion of $H^{p-1}(Z)$ in $H^p(M,\Z)$ at the level of differential forms.  

\begin{proposition}
\label{p:epsZ}
There exists a $G$-invariant closed log $1$-form $\varepsilon_Z$ with support contained in a $G$-invariant tubular neighborhood $\pi_Z \colon U_Z \rightarrow Z$ of $Z$ such that $\tn{res}_Z(\varepsilon_Z)=1$.  Moreover there exists a $G$-invariant smooth non-negative function $h_Z$ with $h_Z^{-1}(0)=Z$, such that $\varepsilon_Z|_{M\backslash Z}=\tfrac{1}{2}\d \log(h_Z)|_{M\backslash Z}$.
\end{proposition}
\begin{proof}
Suppose first for simplicity that there is a smooth function $f$ such that (i) $f^{-1}(0)=Z$, (ii) $\d f|_Z$ does not vanish, (iii) $\d f$ vanishes outside $U_Z$. Then $\varepsilon_Z=\d f/f$, $h_Z=|f|^2$ have the desired properties.
 
A smooth function with the properties (i)--(iii) may not exist, but one can still construct suitable $\varepsilon_Z$, $h_Z$. Identify $U_Z$ with the total space of the normal bundle $\nu(M,Z)$ using a tubular neighborhood embedding.  Choose a fibre metric on $\nu(M,Z)$ and let $r\colon \nu(M,Z)\rightarrow [0,\infty)$ be the radial function.  Let $\chi \in C^\infty(\bR)$ be a smooth monotone non-decreasing function such that $\chi(u)=u$ for $u<\tfrac{1}{2}$ and $\chi(u)=1$ for $u>1$.  The composition $\chi\circ r^2$ is identically $1$ outside a neighborhood of $Z$ in $U_Z\simeq \nu(M,Z)$, hence can be extended by $1$ to a  smooth function $h_Z$ on $M$. Define 
\[ \varepsilon_Z|_{M\backslash Z}=\frac{1}{2}\d\log(h_Z)|_{M\backslash Z}.\]
Let $U\subset Z$ be an open subset such that $\nu(M,Z)|_U\simeq U \times \bR$ trivializes. Then, up to a reparametrization, $r^2|_U=x^2$ where $x \in \bR$ is the fibre coordinate, hence $\varepsilon_Z|_{U_Z\backslash Z}$ is equal to $\d x/x|_{U_Z\backslash Z}$ on the subset of $U_Z\backslash Z$ where $r<1/2$. It follows that $\varepsilon_Z|_{M\backslash Z}$ extends uniquely to a globally defined log $1$-form $\varepsilon_Z$ with the desired properties.
\end{proof}

The inclusion of $H^{p-1}(Z)$ in $H^p(M,\Z)$ is represented at the level of differential forms by wedge product with $\varepsilon_Z$:
\[ \tau \in \Omega^{p-1}(Z) \mapsto \varepsilon_Z \pi_Z^*\tau \in \Omega^p(M,\Z).\]
Recall that $\omega$ is the log symplectic 2-form on $(M,\Z)$. The (degenerate) log 2-form
\begin{equation} 
\label{e:omegabar}
\bar{\omega}=\omega-\sum_Z \varepsilon_Z\pi_Z^*\tn{res}_Z(\omega)-\frac{1}{2}\sum_{Z,W}\varepsilon_Z\varepsilon_W\tn{res}_{Z,W}(\omega)
\end{equation}
has vanishing residues, hence is smooth (here $\tn{res}_{Z,W}(\omega):=\tn{res}_{Z\cap W}\circ \tn{res}_W(\omega) \in \bR$ is locally constant, because taking residues commutes with $\d$ and $\d \omega=0$). And $[\bar{\omega}]\in H^2(M)$ is the image of $[\omega]\in H^2(M,\Z)$ under the Mazzeo-Melrose isomorphism followed by projection to $H^2(M)$.  For $X \in \g$
\begin{equation} 
\label{e:motivatemm}
\iota(X_M)\omega=\iota(X_M)\bar{\omega}+\sum_Z \varepsilon_Z\pair{c_{Z}}{X}.
\end{equation}
Equation \eqref{e:motivatemm} and Proposition \ref{p:epsZ} show that if the modular weight $c_{Z}$ is non-zero, then one expects a Hamiltonian function generating the flow of $X$ to diverge like $\frac{1}{2}\log(h_Z)$ near $Z=h_Z^{-1}(0)$.  In light of this qualitative difference between the cases $c_Z=0$ and $c_Z \ne 0$, it is convenient to make the following definition.

\begin{definition}
\label{d:strata}
Let $\Z_{\ne 0}$ be the subset of hypersurfaces $Z \in \Z$ such that $c_Z \ne 0$. Non-empty intersections $N=Z_1\cap \cdots \cap Z_k$ of elements $Z_1,...,Z_k \in \Z_{\ne 0}$ will be referred to as \emph{strata of} $(M,\Z_{\ne 0})$.
\end{definition}

The discussion above motivates the following definition, which is only a slight re-wording of that in \cite{guillemin2014symplectic, guillemin2014convexity}.

\begin{definition}
\label{d:Hamiltonian}
Let $(M,\Z,\omega)$ be a $G$-equivariant log symplectic manifold, where $\Z$ is a normal crossing divisor. The $G$ action is \emph{Hamiltonian} if there is a smooth $G$-equivariant \emph{momentum map}
\[ \mu \colon M\backslash \cup \Z_{\ne 0} \rightarrow \g^\ast \]
satisfying
\[ \iota(X_M)\omega=-\d \pair{\mu}{X}, \qquad \forall X \in \g,\]
and such that the map
\begin{equation} 
\label{e:mubar}
\bar{\mu}=\mu+\sum_Z \frac{1}{2}\log(h_Z)c_Z 
\end{equation}
extends smoothly to all of $M$.  By construction the pair $(\bar{\omega},\bar{\mu})$ satisfy $\iota(X_M)\bar{\omega}=-\d \pair{\bar{\mu}}{X}$, hence $(M,\bar{\omega},\bar{\mu})$ is a presymplectic Hamiltonian $G$-space.
\end{definition}

\begin{remark}
\label{r:equivclass}
The equations $\d \bar{\omega}=0$, $\iota(X_M)\bar{\omega}=-\d \pair{\bar{\mu}}{X}$ say that $[\bar{\omega}-\bar{\mu}]$ defines a class in the Cartan model for the equivariant cohomology $H^2_G(M,\bR)$ (see for example \cite{MeinrenkenEncyclopedia} for a brief introduction). Assuming $M$ is connected, this equivariant cohomology class is independent of the choices made in the construction ($\varepsilon_Z$, $U_Z$, $\pi_Z$) up to an overall constant shift of $\bar{\mu}$ by an element of $\mf{z}^*$. Indeed the Mazzeo-Melrose theorem shows that $[\bar{\omega}] \in H^2(M,\bR)$ is independent of choices, so given any other $[\bar{\omega}'-\bar{\mu}']$, the difference $\bar{\omega}-\bar{\omega}'=\d \alpha$ is exact. By averaging we may assume $\alpha$ is $G$-invariant. The momentum map equations for $\bar{\mu}$, $\bar{\mu}'$ imply $\pair{\bar{\mu}-\bar{\mu}'}{X}-\iota(X_M)\alpha$ is constant for any $X \in \g$, hence must be of the form $\pair{\xi}{X}$ for some $\xi \in \g^*$.  By $G$-invariance $\xi \in \z^*$. 
\end{remark}

\begin{example}[2-sphere]
\label{ex:2sphere2}
Revisiting $M=S^2$ from Example \ref{ex:2sphere1}, the log symplectic form $\omega=\d z \d \theta/2\pi (z-a)$, $a \in (-1,1)$ is invariant under the action of $S^1$ by rotation about the $z$ axis with generating vector field $X_M=-2\pi \partial/\partial \theta$. The modular weight is
\[ \pair{c_Z}{X}=\iota(X_Z)\tn{res}_Z(\omega)=-1.\]
The $S^1$ action is Hamiltonian with momentum map $\mu(z,\theta)=-\log(|z-a|)+a'$ defined on $M\backslash \{z=0\}$, for any $a'\in \bR$.  
\end{example}

\begin{remark}
In \cite{guillemin2014convexity} it is shown that under certain conditions the modular weights are highly constrained. In particular assuming $M$ is connected, $\Z$ has no crossings, and that each positive codimension symplectic leaf of the Poisson structure $\omega^{-1}$ is compact, then the modular weights are either all zero or all non-zero and parallel.  In the normal crossing case there can be mixtures of zero and non-zero modular weights (see Section 6.1 in \cite{guillemin2014convexity} for an example).
\end{remark}

\subsection{Products and minimal coupling.}\label{ssec:mincoupling}
The category of log symplectic manifolds is closed under taking products, and more generally under a minimal coupling construction that we briefly outline here (cf. \cite{SymplecticTechniques,SjamaarLerman} for the symplectic case).

Let $\pi \colon P \rightarrow B$ be a principal bundle with compact connected structure group $K$.  Suppose $(B,\Z_B,\omega_B)$ is log symplectic.  Let $\Z=\pi^{-1}(\Z_B)$ be the pullback normal crossing divisor. Let $\sigma \in \Omega^1(P,\k)$ be a connection 1-form, and define a closed log 2-form $\omega_\sigma \in \Omega^2(P\times \k^*,\pr_2^{-1}(\Z))$ by
\begin{equation} 
\label{e:mincoup}
\omega_\sigma=\pi^* \omega_B+\d \pair{\pr_2}{\sigma}.
\end{equation}
Then $\omega_\sigma$ is log symplectic on a neighborhood of $P \times \{0\}\subset P \times \k^*$, with the same proof as in the symplectic case. The action of $K$ given by $k\cdot (p,\xi)=(k\cdot p,\Ad_{k^{-1}}^*\xi)$ is Hamiltonian with momentum map $-\pr_2$, and the reduced space at $0$ is $(B,\Z_B,\omega_B)$.

Suppose \eqref{e:mincoup} is non-degenerate on $P\times U_{\k^*}$ where $U_{\k^*}$ is an open neighborhood of $0$ in $\k^*$. Let $(F,\Z_F,\omega_F,\mu_F)$ be a Hamiltonian log symplectic $K$-manifold such that the range of $\mu_F$ is contained in $U_{\k^*}$. Then the reduced space $M$ at $0$ of $P\times U_{\k^*}\times F$ is log symplectic (for log symplectic reduction cf. \cite{LLSS1}), and can be identified with an open subset of the associated bundle $P\times_K F$ (note that as $\mu_F$ need not be defined everywhere, this could be a proper open subset). Suppose that $B$ carries an auxiliary Hamiltonian $G$-action with momentum map $\mu_B$, and that $P$ is a $G$-equivariant principal $K$-bundle with $G$-invariant connection $\sigma$. Then $M$ becomes a Hamiltonian log symplectic $G$-manifold, with momentum map obtained via reduction in stages. The 2-form and momentum map are induced by the following $K$-basic forms
\[ \omega=\pi^* \omega_B+\d \pair{\mu_F}{\sigma}+\omega_F, \quad \pair{\mu}{X}=\pi^* \pair{\mu_B}{X}+\pair{\mu_F}{\sigma(X_P)}, \quad  X \in \g,\]
defined on the open subset of $P\times F$ where $\mu_F$ is defined.
\begin{example}
\label{ex:bundle}
Let $a \in (-1,1)$ and let $B=S^2$ be the unit sphere with divisor $\Z_B=\{\{z=a\}\}$ and $\omega_B=\d z \d \theta/2\pi(z-a)$ as in Example \ref{ex:2sphere1}. Let $\pi\colon P=S^3 \rightarrow B$ be the Hopf fibration over $S^2$ with connection 1-form $\sigma$ such that $\d \sigma=\pi^*\Omega$, where $\Omega=\d z \d \theta/4\pi$. The minimal coupling 2-form \eqref{e:mincoup} is
\[ \omega_\sigma=\pi^*\omega_B+\xi\pi^*\Omega+\d\xi\cdot \sigma=\Big(\frac{1}{z-a}+\frac{\xi}{2}\Big)\frac{\d z \d \theta}{2\pi}+\d\xi\cdot \sigma, \qquad \xi \in \tn{Lie}(S^1)\simeq \bR.\]
This is non-degenerate on the open set $P\times \{\frac{-2}{1+a}<\xi<\frac{2}{1-a}\}\subset P\times \bR$. Let $F=S^2$ be ordinary symplectic with $\omega_F=\Omega$, and $K=S^1$ momentum map $\mu_F(z,\theta)=z$. Since the range $\{-1\le\xi\le 1\}$ of $\mu_F$ is contained in $\{\frac{-2}{1+a}<\xi<\frac{2}{1-a}\}$, the minimal coupling construction applies and produces a log symplectic structure on $M=P\times_{S^1}F$. From Example \ref{ex:2sphere2}, the base $(B,\Z_B,\omega_B)$ admits a Hamiltonian $S^1$-action. Lift this $G=S^1$-action to $P$ using the Kostant condition $\iota(X_P)\sigma=\mu_0X$ where $X \in \g=\bR$ and $\mu_0(z,\theta)=\frac{z}{2}+\frac{1}{2}$ maps $B=S^2$ to the interval $[0,1]$. The minimal coupling construction turns $M$ into a log symplectic Hamiltonian $S^1$-space.
\end{example}

\subsection{Prequantum data.}\label{ss:prequantum}
Following \cite[Chapter 6]{HamiltonianCobordismBook}, we will say that a presymplectic Hamiltonian $G$-space $(M,\bar{\omega},\bar{\mu})$ is \emph{prequantizable} if there exists a $G$-equivariant complex line bundle $L$ such that the image of its equivariant first Chern class $c_1^G(L)$ in $H_G^2(M,\bR)$ is $[\bar{\omega}-\bar{\mu}]$.  By \cite[Proposition 6.11]{HamiltonianCobordismBook}, one can equip $L$ with a Hermitian structure and Hermitian connection $\nabla^L$ such that the first Chern form of $(L,\nabla^L)$ is $\bar{\omega}$ and such that the infinitesimal $\g$-action $\rho^L$ on $\Gamma(L)$ satisfies Kostant's formula:
\begin{equation} 
\label{e:Kostant}
\rho^L(X)=\nabla^L_{X_M}+2\pi \sqrt{-1} \pair{\bar{\mu}}{X} 
\end{equation}
In the opposite direction, given $(L,\nabla^L)$ with $(\nabla^L)^2=-2\pi \sqrt{-1} \bar{\omega}$, one may lift the $\g$-action on $M$ to $L$ using \eqref{e:Kostant}, and then the condition requires that the infinitesimal action integrates to a $G$-action.  The data $(L,\nabla^L)$ is referred to as \emph{prequantization data} for $(M,\bar{\omega},\bar{\mu})$.

Let $(M,\Z,\omega)$ be a log symplectic manifold. Recall that the Mazzeo-Melrose isomorphism determines a map $H^2(M,\Z)\rightarrow H^2(M)$ which sends $[\omega]$ to $[\bar{\omega}]$, where $\bar{\omega}$ is constructed as in \eqref{e:omegabar}. One may furthermore construct $\bar{\mu}$ as in equation \eqref{e:mubar} and by Remark \ref{r:equivclass} the class $[\bar{\omega}-\bar{\mu}]+\z^* \in H^2_G(M,\bR)/\z^*$ is independent of choices.
\begin{definition}
\label{d:prequantizable}
$(M,\Z,\omega,\mu)$ is \emph{prequantizable} if the image of $[\bar{\omega}-\bar{\mu}]+\z^*\in H^2_G(M,\bR)/\z^*$ equals the image of the equivariant $1^{st}$ Chern class $c_1^G(L)$ of a $G$-equivariant complex line bundle $L$. In this case, after a shift of $\mu$, $\bar{\mu}$ by an element of $\z^*$, $(M,\bar{\omega},\bar{\mu})$ is prequantizable as a presymplectic Hamiltonian $G$-space. We define \emph{prequantization data} $(L,\nabla^L)$ for $(M,\Z,\omega,\mu)$ to be prequantization data for the presymplectic Hamiltonian $G$-space $(M,\bar{\omega},\bar{\mu})$.
\end{definition}
\begin{remark}
\label{r:prequantizable}
Suppose $G=T$ is a torus, $M$ is connected and $M^T\ne \emptyset$. Then according to \cite[Example 6.10]{HamiltonianCobordismBook}, $(M,\bar{\omega},\bar{\mu})$ is prequantizable iff $[\bar{\omega}]$ is integral and $\bar{\mu}(p)$ lies in the weight lattice of $T$, for some $p \in M^T$.
\end{remark}

\begin{example}[2-sphere]
\label{ex:2sphere3}
Returning to Example \ref{ex:2sphere2}, the regularized form $\bar{\omega}$ is integral if its integral over $S^2$ is an integer $n$. The latter equals the principal value integral of $\omega=\d z \d \theta/2\pi (z-a)$ over $S^2$, which is
\[ n=\tn{PV}\int_{-1}^1 \frac{\d z}{z-a}=\log \left|\frac{1-a}{1+a}\right|.\]
This forces the parameter $a$ to lie in a countable subset of $(-1,1)$. Recall that the momentum map $\mu(z,\theta)=-\log(|z-a|)+a'$ where $a' \in \bR$. For any pair of integers $n_1,n_2$ such that $n_2-n_1=n$, $a'$ can be chosen such that $\mu$ takes the values $n_1,n_2$ at the fixed-points $z=1,-1$ respectively. We can arrange that $\bar{\mu}$ agrees with $\mu$ at the fixed points, hence the lift \eqref{e:Kostant} integrates to an action of $S^1$.
\end{example}

\subsection{Spin$_c$ structure.}
Let $(M,\Z,\omega)$ be a log symplectic manifold where $(M,\Z)$ is a normal crossing divisor.  The log $2$-form $\omega$ can be regarded as a fibre-wise symplectic form on the log tangent bundle $T_\Z M$, i.e. $(T_\Z M,\omega)$ is a symplectic vector bundle. On any symplectic vector bundle $(V,\omega_V) \rightarrow M$, one can find a \emph{compatible complex structure}, which is by definition a complex structure $J \in \End(V)$, $J^2=-1$ such that $g_V(\cdot,\cdot)=\omega_V(\cdot,J\cdot)$ is a positive definite symmetric bilinear form on $V$.  The space of compatible complex structures is contractible, hence in particular the choice of such a $J$ is unique up to homotopy.  This leads to the following corollary of Theorem \ref{t:stableiso}.

\begin{corollary}
\label{c:spinc}
Let $(M,\Z,\omega)$ be a log symplectic manifold such that $\Z$ admits global defining functions.  Then $M$ admits a stable almost complex structure. After choosing orientations on the line bundles $\I_Z$ for $Z \in \Z$, the construction is canonical up to homotopy.
\end{corollary}
\begin{proof}
By the remarks preceding the statement of the corollary, there is a compatible complex structure on the symplectic vector bundle $T_\Z M$, which is unique up to homotopy. By Theorem \ref{t:stableiso}, $\ul{\bR}^2\oplus TM\simeq \ul{\bR}^2\oplus T_\Z M$, and the isomorphism is determined up to homotopy by choices of orientations on the line bundles $\I_Z$ for $Z \in \Z$. Using the isomorphism, we transfer the complex structure on $\ul{\bR}^2\oplus T_\Z M$ to $\ul{\bR}^2\oplus TM$.
\end{proof}
\begin{remark}
If $(M,\Z,\omega)$ is oriented ($\det(TM)$ is trivial) and log symplectic ($\det(T_\Z M)$ is trivial), then by \eqref{e:detoz}, $\I_\Z$ admits a global non-vanishing section. Assuming furthermore that $\Z$ does not have any crossings, a global non-vanishing section of $\I_\Z$ is the same thing as a global defining function for $\Z$.
\end{remark}
\begin{remark}
It was pointed out to us that in the case of a single hypersurface a very similar result appears in \cite[Section 4]{guillemin2000unfolding}, in the context of folded symplectic structures.
\end{remark}

A stable almost complex structure determines a Spin$_c$ structure (see for example \cite[Appendix D]{HamiltonianCobordismBook}, \cite[Section 5]{guillemin2000unfolding}). A somewhat different direct construction of this Spin$_c$ structure was outlined in \cite[Remark A.11]{BLSbsympl}.  Choosing a Riemannian metric on $M$, it is convenient to think of the Spin$_c$ structure in terms of a corresponding spinor module, that is, $\bZ_2$-graded hermitian vector bundle $S \rightarrow M$, whose fibres $S_m$ form a smooth family of irreducible modules for the family of Clifford algebras $\bC l(T_mM)$, i.e. there is given an isomorphism $c \colon \bC l(TM)\rightarrow \End(S)$.  The $\bZ_2$-grading on $S$ is the eigenspace decomposition for the chirality operator $c(\Gamma)$, where $\Gamma=(\sqrt{-1})^{n/2} e_1\cdots e_n$ in terms of a local oriented orthonormal frame $e_1,...,e_n$.  Away from $\Z$, the anchor map is an isomorphism, and this may be used to construct an isomorphism of $\bZ_2$-graded complex vector bundles
\begin{equation} 
\label{e:MZiso}
S|_{M\backslash \cup\Z}\simeq \wedge T^{1,0}(M\backslash \cup\Z).
\end{equation}
On the RHS the $\bZ_2$-grading is given by the chirality operator determined by the orientation on $TM$ and \emph{not} by the symplectic orientation on $T_\Z M$; put differently \eqref{e:MZiso} maps the even subbundle $S^+$ to forms of even (resp. odd) degree over components of $M\backslash \cup \Z$ where the orientations of $TM$, $T_\Z M$ agree (resp. disagree).

\begin{remark}
This slight subtlety concerning the $\bZ_2$-grading is important. It is closely related to the `cancellations' noted between the quantizations of pairs of reduced spaces in a neighborhood of the hypersurface in \cite{guillemin2018geometric}.
\end{remark}

A spinor module $E$ has an `anti-canonical line bundle' defined by
\[ \L=\Hom_{\bC l(TM)}(E^*,E).\]
This is a complex line bundle (by Schur's lemma). In case $E$ arises from a stable almost complex structure, say, from a complex structure on $\ul{\bR}^p\oplus TM$,
\[ \L=\tn{det}_\bC(\ul{\bR}^{p}\oplus TM).\]
In our situation, $E=S$ arises from a complex structure on $\ul{\bR}^2\oplus TM \simeq \ul{\bR}^2\oplus T_\Z M$, where $T_\Z M$ carries an $\omega$-compatible complex structure, and $\ul{\bR}^2\simeq \ul{\bC}$ carries its standard complex structure.  Since $\det_\bC(\ul{\bC})$ is trivial, we deduce that in our setting
\begin{equation} 
\label{e:detline}
\L\simeq \tn{det}_\bC(T_\Z M).
\end{equation}

\subsection{Quantization.}
A choice of connection $\nabla$ on a spinor module $E$ determines a Spin$_c$ Dirac operator $\dirac$, defined by the composition
\[ \Gamma(E)\xrightarrow{\nabla} \Gamma(T^*M\otimes E)\xrightarrow{g^\#} \Gamma(TM\otimes E)\xrightarrow{c} \Gamma(E).\]
In our case the spinor module of interest is $E=S\otimes L$, where $S$ is the spinor module associated to the stable almost complex structure as in Corollary \ref{c:spinc}, and $L$ is a prequantum line bundle. We will denote the resulting operator $\dirac^L$; it is an odd elliptic operator acting on smooth sections of a $\bZ_2$-graded Hermitian vector bundle $S\otimes L$ over the Riemannian manifold $M$.  If $M$ is compact, $\dirac^L$ has a well-defined Fredholm index, denoted $\index(\dirac^L)$, which is independent of the choices of metrics and connections.

\begin{definition}
\label{d:RRnumber}
Let $(M,\Z,\omega,L)$ be a compact prequantized log symplectic manifold, where $\Z$ is a normal crossing divisor that admits global defining functions.  Let $S$ be the spinor module obtained from the stable almost complex structure on $M$ and choices of orientations on the line bundles $\I_Z$ for $Z \in \Z$.  We define the \emph{quantization} or \emph{Riemann-Roch number} of $(M,\Z,\omega,L)$ to be the index of the Dirac operator $\index(\dirac^L)\in \bZ$.  In the $G$-equivariant case, we may take the equivariant index $\index_G(\dirac^L)\in R(G)$, the representation ring of $G$.
\end{definition}

The index theorem gives the following formula for the Riemann-Roch number:
\begin{equation} 
\label{e:indthm}
\index(\dirac^L)=\int_M \Ahat(TM)\Ch(\L)^{1/2}\Ch(L)=\int_M \Td(T_\Z M)\Ch(L).
\end{equation}
The second expression follows from \eqref{e:detline} and Corollary \ref{c:Pont}.  In the equivariant case, note that since $G$ preserves $\Z$, the fixed-point submanifold $M^g$ of $g \in G$ is automatically transverse to the image of the anchor map $T_\Z M \rightarrow TM$.  It follows that each fixed-point component $F \subset M^g$ acquires a normal crossing divisor $\Z_F=\{Z \cap F|Z \in \Z\}$ admitting global defining functions, and the anchor descends to a $g$-equivariant isomorphism
\[ T_\Z M|_F/T_{\Z_F}F \simeq \nu_F \]
where $\nu_F$ is the normal bundle to $F$ in $M$.  At each point $m \in F$, $T_\Z M|_m$ is a complex representation of the cyclic subgroup generated by $g$, hence the fixed subspace $T_{\Z_F}F|_m$ and the quotient $\nu_F|_m$ are both complex. One has the following version of the fixed-point formula
\begin{equation}
\label{e:indthm2}
\index_G(\dirac^L)(g)=\sum_{F \subset M^g}\int_F \frac{\Td(T_{\Z_F}F)\Ch^g(L)}{\Ch^g\big(\lambda_{-1}\nu^{0,1}_F\big)},
\end{equation}
where $\nu_F^{0,1}\subset \nu_F \otimes \bC$ is the $-\sqrt{-1}$-eigenbundle for the complex structure, and $F$ carries the orientation induced from the orientations on $M$, $\nu_F$.

\begin{remark}
\label{r:dependenceonorientation}
As a consequence of \eqref{e:indthm}, \eqref{e:indthm2}, the quantization of $(M,\Z,\omega,L)$ depends only on the product orientation on $\I_\Z$ (not on the individual orientations of the $\I_Z$), or equivalently, by \eqref{e:detoz}, on the induced orientation on $M$. Reversing the orientation reverses the sign of the quantization.
\end{remark}

\begin{example}
\label{ex:2sphere4}
Let $M=S^2$ be the unit sphere in $\bR^3$ equipped with the log symplectic form $\omega=\d z \d \theta/2\pi (z-a)$, $a \in (-1,1)$, $\Z=\{\{z=a\}\}$ and Hamiltonian $S^1$ action with generating vector field $X_M=-2\pi \partial/\partial \theta$ and momentum map $\mu=-\log(|z-a|)+a'$ (Examples \ref{ex:2sphere1}, \ref{ex:2sphere2}). From Example \ref{ex:2sphere3}, $\omega$ is integral if $n=\log|(1-a)/(1+a)| \in \bZ$, and in this case for any $n_1,n_2 \in \bZ$ with $n_2-n_1=n$ there is a choice of $a'$ such that $\mu|_{z=1}=n_1$, $\mu|_{z=-1}=n_2$.  Applying the Atiyah-Bott formula \eqref{e:indthm2}, the $S^1$-equivariant quantization is
\begin{equation} 
\label{e:AB}
\index_{S^1}(\dirac^L)(t)=\frac{t^{n_1}}{1-t}-\frac{t^{n_2}}{1-t} \in \bZ[t,t^{-1}]=R(S^1).
\end{equation}
Note that in the more familiar symplectic setting, the weights for the $S^1$ action at the fixed points would have opposite signs. Here the signs are the same since we use a complex structure compatible with the log symplectic form, and the latter undergoes a sign flip across the hypersurface $z=a$. The contribution from $z=-1$ has an overall minus sign from the induced orientation. The result is a (virtual) representation of $S^1$ of dimension $n$.  
\end{example}
\begin{example}[Minimal coupling]
\label{e:mincoupling}
The minimal coupling construction from Section \ref{ssec:mincoupling} leads to further examples. Suppose the base $(B,\omega_B)$ and fibre $(F,\omega_F)$ both satisfy the conditions in Definition \ref{d:RRnumber}. Let us further assume that the momentum map $\mu_F$ of the fibre is everywhere defined. Then the minimal coupling space is the full associated bundle $M=P\times_K F$, and satisfies the conditions in Definition \ref{d:RRnumber} as well. Let $Q_F=Q_{F,+}-Q_{F,-} \in R(K)$ be the quantization of the fibre, where $Q_{F,\pm}$ are representations of $K$ (the $\bZ_2$-graded components of the kernel of $\dirac^{L_F}$), and let $\Q_{F,\pm}=P\times_K Q_{F,\pm}$ be the associated bundles. The quantization of $M$ is the difference
\begin{equation} 
\label{e:fibreindex}
\index_G(\dirac^{L_B\otimes \Q_{F,+}})-\index_G(\dirac^{L_B\otimes \Q_{F,-}}),
\end{equation}
where $L_B \rightarrow B$ is the chosen prequantum line bundle on the base, and $\dirac^{L_B}$ is the Dirac operator on the base. There are various approaches to \eqref{e:fibreindex}. One approach is to use the Atiyah-Singer theorem for families, which implies, by comparing index formulas, that the index of the Dirac operator on $M$ equals the index of the Dirac operator on $B$ twisted by the class $[\dirac^{M/B}]\in \tn{K}^0_G(B)$ determined by the family of Dirac operators on the fibres of $M \rightarrow B$. In this case the $\bZ_2$-graded components of the kernel of the family have constant dimension, forming the vector bundles $\Q_{F,+}$, $\Q_{F,-}$ respectively, hence $[\dirac^{M/B}]=[\Q_{F,+}]-[\Q_{F,-}]\in \tn{K}^0_G(B)$, and \eqref{e:fibreindex} follows. For slightly more abstract approaches, see for example \cite[Remark 3.7]{BaumHigsonSchick} or \cite[Theorem 3.5]{AtiyahTransEll} (in the latter, one needs to lift $\dirac^{L_B}$ to a $K=H$-transversely elliptic operator on $\tn{tot}(P)$, and take $H$-invariants on both sides of the equation).   

For the case of the $S^2$-bundle over $S^2$ described in Example \ref{ex:bundle}, the quantization of the fibre is $Q_{F,+}=\bC_{-1}\oplus\bC_0\oplus \bC_1$, $Q_{F,-}=0$. For $j \in \bZ$ let $\O(j)=P\times_{S^1}\bC_j \rightarrow B=S^2$, hence $\Q_{F,+}=\O(-1)\oplus\O\oplus \O(1)$, $\Q_{F,-}=0$. Suppose the parameter $a$ is chosen as in Example \ref{ex:2sphere3} such that $n=\log(|1-a|/|1+a|)$ is an integer. Then $(B,\Z_B,\omega_B)$ satisfies the prequantization condition and the prequantum line bundle is $L_B=\O(n)$. As in Example \ref{ex:2sphere3} fix $n_1,n_2\in \bZ$ such that $n=n_2-n_1$, and choose $a' \in \bR$ such that
\[ -\log(|1-a|)+a'=n_1, \qquad -\log(|-1-a|)+a'=n_2.\]
The lift of the $S^1$ action to $L_B$ is determined by the momentum map
\[ \mu_B(z,\theta)=-\log(|z-a|)+a',\]
and in particular the weights of the $S^1$ action on $L_B|_{z=\pm 1}$ are $n_1,n_2$ respectively. The weights of the induced $S^1$ action on $\Q_{F,+}$ at $z=\pm 1$ are $\{-1,0,1\}$ and $\{0,0,0\}$ respectively. The $S^1$-equivariant quantization of $M$ is given by the Atiyah-Bott fixed-point formula:
\[ \index_{S^1}(\dirac^{L_B\otimes \Q_{F,+}})=\frac{t^{n_1-1}+t^{n_1}+t^{n_1+1}}{1-t}-\frac{3t^{n_2}}{1-t} \in \bZ[t,t^{-1}]=R(S^1).\]
In fact $M$ admits a Hamiltonian $S^1\times S^1$ action (the $S^1$ action on the fiber $F$ induces a second $S^1$ action on $M$) and becomes an example of a toric log symplectic manifold; see Section \ref{s:toric} for further discussion of the toric case.
\end{example}

\section{Non-abelian localization and $[Q,R]=0$}
In this section we turn to the non-abelian localization formula and the $[Q,R]=0$ theorem.  We adapt an approach due to Paradan \cite{ParadanRiemannRoch} in the compact symplectic setting, based on deformation through transversely elliptic symbols.  Throughout this section we assume $(M,\Z,\omega,\mu)$ is a compact Hamiltonian log symplectic manifold.  

We will furthermore assume that the momentum map $\mu \colon M\backslash\cup \Z_{\ne 0}\rightarrow \g^*$ is proper (see also Lemma \ref{l:proper}).  Although this condition was not needed to make sense of Definition \ref{d:RRnumber}, it becomes relevant in formulating and proving the $[Q,R]=0$ theorem.  In particular if $\mu$ is not proper, then $\mu^{-1}(0)/G$ may be non-compact, in which case Definition \ref{d:RRnumber} does not apply to the reduced space.

\begin{example}
For a simple example where $\mu^{-1}(0)$ is non-compact, consider $M=S^2\times S^2 \circlearrowleft S^1$ with
\[ \omega=\frac{\d z_1}{2\pi z_1}\d \theta_1-\frac{\d z_2}{2\pi z_2}\d \theta_2, \quad X_M=-2\pi \frac{\partial}{\partial \theta_1}-2\pi \frac{\partial}{\partial \theta_2}, \quad \mu=-\log(|z_1|)+\log(|z_2|),\]
using notation as in Example \ref{ex:2sphere2}.  The momentum map is defined on the subset $M\backslash \cup \Z=\{z_1z_2\ne 0\}$. The fibre $\mu^{-1}(0)$ is the intersection of this subset with $\{|z_1|=|z_2|\}$. 
\end{example}

\begin{remark}
If $M$ is compact and the divisor $\Z$ does not have any crossings, then the momentum map $\mu \colon M\backslash \cup\Z_{\ne 0}\rightarrow \g^*$ is automatically proper.
\end{remark}

\subsection{Kirwan vector field.}
Choose an invariant inner product on $\g$ that we use to identify $\g\simeq \g^\ast$. Let $T$ be a maximal torus with Lie algebra $\t\simeq \t^*$, and let $\t_+^*$ be a positive chamber.

\begin{definition}
The \emph{Kirwan vector field} is the $G$-invariant vector field $\kappa$ on $M\backslash \cup \Z_{\ne 0}$ given by the formula
\begin{equation} 
\label{e:Kirwan}
\kappa(m)=\big(\mu(m)\big)_M(m), \quad m \in M\backslash \cup \Z_{\ne 0}.
\end{equation}
Equivalently $\kappa$ is the Hamiltonian vector field for the function $-\|\mu\|^2/2$ on $M \backslash \cup \Z_{\ne 0}$. 
\end{definition}

\begin{proposition}
\label{p:finite}
Let $(M,\Z,\omega,\mu)$ be a compact Hamiltonian log symplectic manifold. The vanishing locus of the Kirwan vector field $\kappa$ is
\begin{equation} 
\label{e:decomposeB}
\C=G\cdot \bigcup_{\beta \in \t^*_+}M^\beta \cap \mu^{-1}(\beta).
\end{equation}
The set $\B\subset \t^*_+$ of $\beta$ such that $M^\beta \cap \mu^{-1}(\beta)\ne \emptyset$ is finite.
\end{proposition}
\begin{proof}
The description \eqref{e:decomposeB} follows immediately from the definition of $\kappa$ (cf. \cite{Kirwan}).  Since $M$ is compact, the set $\tn{St}(M,\t)$ of infinitesimal stabilizer types for the $T$ action is finite. If $\B$ is infinite, then one of the $M^{\h}$ with $0\ne \h \in \tn{St}(M,\t)$ would have to occur infinitely many times in the list $(M^\beta)_{\beta \in \B}$.  Hence the set $\B \cap \h$ would be infinite.  The momentum map condition implies that the projection $\mu_{\h^*}$ of the momentum map to $\h^*$ is locally constant on $M^\h$.  Therefore for $\beta \in \B \cap \h$, $\mu_{\h^*}$ must take the value $\beta$ on the components of $M^\h$ intersecting $M^\h \cap \mu^{-1}(\beta)$ non-trivially.  Since $M$ is compact, $M^\h$ has finitely many components, and this is a contradiction. 
\end{proof}
\begin{remark}
The set $\C$ in \eqref{e:decomposeB} coincides with the critical locus of the Hamiltonian $-\|\mu\|^2/2$.
\end{remark}

In Paradan's approach \cite{ParadanRiemannRoch}, one uses the Kirwan vector field to deform the symbol of the Dirac operator in the space of transversely elliptic symbols. In our situation $\kappa$ is only defined on $M\backslash \cup \Z_{\ne 0}$. We now explain a straight-forward method of modifying $\kappa$ so that it extends smoothly to $M$, in such a way that the vanishing locus \eqref{e:decomposeB} is unchanged.

\begin{definition}
Fix $\epsilon>0$ and let $\log_\epsilon \colon [0,\infty) \rightarrow [\log(\epsilon),\infty)$ be a smooth monotone non-decreasing modification of the function $\log$, such that $\log_\epsilon(|x|)=\log(|x|)$ for $|x|>2\epsilon$, $\log_\epsilon(|x|)=\log(\epsilon)$ for $|x|<\epsilon$.
\end{definition}
\begin{definition}
\label{d:tildevariables}
Fix $\epsilon>0$ and define
\[\ti{\mu}=\bar{\mu}+\sum_{Z \in \Z} \frac{1}{2}\log_\epsilon(h_Z)c_Z. \]
In other words we have replaced $\log$ with $\log_\epsilon$ in \eqref{e:mubar}.  Similarly define the degenerate smooth 2-form $\ti{\omega}$ by replacing $\varepsilon_Z$ with $\d \frac{1}{2}\log_\epsilon(h_Z)$ in \eqref{e:omegabar} (recall $\varepsilon_Z\upharpoonright M\backslash Z=\d \frac{1}{2}\log(h_Z)\upharpoonright M\backslash Z$). The pair $(\ti{\omega},\ti{\mu})$ is smooth on all of $M$ and satisfies the momentum map equation $\iota(X_M)\ti{\omega}=-\d\pair{\ti{\mu}}{X}$.  Let $\ti{\kappa}$ be the $G$-equivariant vector field on $M$ defined as in \eqref{e:Kirwan} but with $\ti{\mu}$ in place of $\mu$. We will write $\ti{\mu}_\epsilon$, $\ti{\kappa}_\epsilon$ instead of $\ti{\mu}$, $\ti{\kappa}$ when we want to emphasize the dependence on $\epsilon$.
\end{definition}

The vanishing locus of $\ti{\kappa}$ is
\begin{equation} 
\label{e:tildekappa}
\ti{\C}=G\cdot \bigcup_{\beta \in \t^*_+} M^\beta \cap \ti{\mu}^{-1}(\beta).
\end{equation}
(Note however that since $\ti{\omega}$ is degenerate, the vanishing locus of $\ti{\kappa}$ is no longer the same as the critical locus for $-\|\ti{\mu}\|^2/2$.) We will argue that for $\epsilon$ sufficiently small, the vanishing loci \eqref{e:decomposeB}, \eqref{e:tildekappa} coincide.

\begin{lemma}
\label{l:proper}
Let $(M,\Z,\omega,\mu)$ be a compact Hamiltonian log symplectic space.  Then $\mu\colon M\backslash \cup \Z_{\ne 0}\rightarrow \g^*$ is proper if and only if for each stratum $N=Z_1\cap \cdots \cap Z_k$, $Z_j \in \Z_{\ne 0}$, the cone generated by non-negative linear combinations of the modular weights $c_{Z_1},...,c_{Z_k}$ is strongly convex.
\end{lemma}
\begin{proof}
If $c_{Z_1},...,c_{Z_k}$ do not form a convex cone, then some non-zero non-negative linear combination of them vanishes.  Then using \eqref{e:mubar} and since $\bar{\mu}$ is bounded on $M$, it will be possible to choose a sequence of points $p_n \in M\backslash \cup \Z_{\ne 0}$ approaching $N$ such that $|\mu(p_n)|$ remains bounded. Hence $\mu$ is not proper.  The other direction is similar, again using equation \eqref{e:mubar}.
\end{proof}

\begin{lemma}
Let $N=Z_1\cap \cdots \cap Z_k$ be a stratum of $(M,\Z_{\ne 0})$.  There is an open subset $\mf{U}_N\subset \g^*$ containing $\tn{span}\{c_{Z_1},...,c_{Z_k}\} \backslash \{0\}$ and invariant under non-zero scalar multiplication, as well as an open neighborhood $U_N$ of $N$ in $M$ such that if $X \in \mf{U}_N$ then $X_M$ does not vanish on $U_N$. 
\end{lemma}
\begin{proof} 
Since non-vanishing is an open condition (and invariant under non-zero scalar multiplication of the generator $X$), it suffices to consider the vector field generated by some $X \in \tn{span}\{c_{Z_1},...,c_{Z_k}\} \backslash \{0\}$ on $N$.  Let
\[ X=\sum_j t_j c_{Z_j}\]
be non-zero.  By Definition \ref{d:modwt},
\[ \iota(X_{Z_i})\tn{res}_{Z_i}(\omega)=\sum_j t_j\pair{c_{Z_i}}{c_{Z_j}}=:r_i \in \bR. \]
The constants $r_i$, $i=1,...,k$ cannot all be zero because
\[ \sum_i t_i r_i=\sum_{i,j} t_it_j\pair{c_{Z_i}}{c_{Z_j}}=\|X\|^2 > 0.\]
We conclude that $X_M$ does not vanish on the intersection $N$, since its contraction with some residue of $\omega$ does not vanish.
\end{proof}

\begin{lemma}
\label{l:loci}
Assume $\mu \colon M\backslash \cup \Z_{\ne 0}\rightarrow \g^*$ is proper.  Then there is an open neighborhood $U''$ of $\cup \Z_{\ne 0}$, and an $\epsilon''>0$ such that for all $\epsilon''>\epsilon>0$, $\ti{\kappa}_\epsilon$ does not vanish on $U''$. 
\end{lemma}
\begin{proof}
We will use the notation from the previous lemma. Let $Z_1\cap \cdots \cap Z_k=N\in \N$ range over strata of $(M,\Z_{\ne 0})$.  We may choose an open cover $\{U_N'\}_{N \in \N}$ of $\cup \Z_{\ne 0}$ such that (i) $U_N'\subset U_N$, (ii) for all $Z \in \Z_{\ne 0}$ not intersecting $N$, the function $|h_Z|\upharpoonright U_N'\ge \delta>0$.

Recall
\[ \ti{\mu}_\epsilon=\bar{\mu}-\sum_Z \frac{1}{2}\log_\epsilon(h_Z)c_Z,\]
where $\bar{\mu}$ extends smoothly to $M$, hence is bounded. By properness $c_{Z_1},...,c_{Z_k}$ generate a strongly convex cone.  If $\epsilon$ is small, it follows that in $U_N'$ and sufficiently near $N$, $\ti{\mu}_\epsilon$ is a small perturbation of
\[ -\sum_j \frac{1}{2}\log(\epsilon)c_{Z_j}.\]
More precisely there is an open set $U_N'' \subset U_N'$ with $U_N''\cap N=U_N' \cap N$, such that $\ti{\mu}_\epsilon\upharpoonright U_N''$ takes values in the open set $\mf{U}_N \subset \g^*$. (We emphasize that the fact that $c_{Z_1},...,c_{Z_k}$ generate a strongly convex cone is being used here.)  The previous lemma implies that $\ti{\kappa}_\epsilon$ does not vanish on $U_N''$.  There are finitely many $N$ to consider, hence taking $\epsilon$ sufficiently small, we can ensure this holds for all $N \in \N$.  Then take $U''=\cup_{N \in \N} U_N''$.  
\end{proof}

\begin{corollary}
\label{c:loci}
Assume $\mu \colon M\backslash \cup \Z_{\ne 0}\rightarrow \g^*$ is proper. For $\epsilon>0$ sufficiently small the vanishing loci $\C=\ti{\C}$, and $\mu$, $\ti{\mu}$ agree on $\C$.
\end{corollary}
\begin{proof}
Since
\[ \mu=\bar{\mu}-\sum_Z \frac{1}{2}\log(h_Z)c_{Z}, \]
we see that by choosing $\epsilon$ sufficiently small, we may ensure $\mu$, $\ti{\mu}_\epsilon$ agree everywhere except on an arbitrarily small neighborhood of $\cup \Z_{\ne 0}$. In particular we may choose $\epsilon''>\epsilon>0$ sufficiently small that $\mu$, $\ti{\mu}_\epsilon$ agree on $M\backslash U''$, where $U''$, $\epsilon''$ are from Lemma \ref{l:loci}.  The result follows.
\end{proof}

From now on we assume $\epsilon$ is as in Corollary \ref{c:loci}, hence $\ti{\kappa}$ is a smooth vector field on $M$ with vanishing locus $\C$ (equation \eqref{e:decomposeB}).

\subsection{Paradan-type deformation.}
Having constructed a suitable $G$-equivariant map $\ti{\mu}\colon M\rightarrow \g^*\simeq \g$ and associated vector field $\ti{\kappa}$, we will now apply the general results of Paradan \cite{ParadanRiemannRoch} and Paradan-Vergne \cite{WittenNonAbelian} to deduce a $[Q,R]=0$ theorem.

\begin{definition}
Assume we are in the setting of Definition \ref{d:RRnumber}. We will identify $TM\simeq T^*M$ using the Riemannian metric. Let $\xi \in T^*M$.  The symbol of the Dirac operator $\dirac^L$ is $\sigma(\xi)=c(\xi)$, Clifford multiplication on the pullback of the spinor module $S\otimes L$ to $T^*M$.  Using the vector field $\ti{\kappa}$ we define a deformed symbol
\[ \ti{\sigma}(\xi)=c(\xi-\ti{\kappa}).\]
\end{definition}
The symbols $\sigma,\ti{\sigma}$ are homotopic (through the family $t \mapsto c(\xi-t\ti{\kappa})$, $t \in [0,1]$), hence define the same class in the K-theory $\tn{K}_G^0(T^*M)$.  Following Paradan \cite{ParadanRiemannRoch} we will consider $\ti{\sigma}$ as a transversely elliptic symbol, in order to take advantage of the extra flexibility allowed for homotopies of such symbols.

Let $Y$ be a (possibly non-compact) $G$-manifold, and let $T_G^*Y\subset T^*Y$ be the conormal space to the $G$-orbit directions. The compactly-supported $G$-equivariant K-theory $\tn{K}^0_G(T_G^*Y)$ can be described in terms of equivalence classes of pairs $(E,\gamma)$, where $E$ is a $\bZ_2$-graded $G$-equivariant vector bundle over $T_G^*Y$, and $\gamma \in \End(E)$ is a $G$-equivariant odd bundle endomorphism that restricts to an isomorphism outside a compact set. Atiyah \cite{AtiyahTransEll} defined an analytic index map
\begin{equation} 
\label{e:anind}
\index_G \colon \tn{K}^0_G(T_G^*Y)\rightarrow R^{-\infty}(G),
\end{equation}
extending the more familiar analytic index map $\tn{K}^0_G(T^*Y)\rightarrow R(G)$, where $R^{-\infty}(G)$ is the formal completion of the representation ring, the set of possibly infinite formal integer linear combinations of irreducible characters $\chi_\lambda$ where $\lambda$ is a dominant weight.  

The restriction of $\ti{\sigma}$ to $T_G^*M$ determines a class $[\ti{\sigma}]\in \tn{K}^0_G(T^*_GM)$. Since the analytic index map \eqref{e:anind} extends the ordinary analytic index map, we have
\begin{equation}
\label{e:extendind}
\index_G(\dirac^L)=\index_G([\ti{\sigma}]).
\end{equation}

Since $c(\xi-\ti{\kappa})$ is invertible except when $\xi=\ti{\kappa}(m)$, the subset of $T_G^*M$ where $\ti{\sigma}$ fails to be invertible is $\C=\ti{\kappa}^{-1}(0_M) \subset T_G^*M$, the vanishing locus of $\ti{\kappa}$ viewed as a subset of the zero section in $T^*_GM$.  Atiyah \cite{AtiyahTransEll} established an excision-type property of the index that applies in this situation. Let $U_\beta$ be a $G$-invariant open neighborhood of $\C_\beta=G\cdot (M^\beta \cap \mu^{-1}(\beta))\subset \C$. We may assume the $U_\beta$, $\beta \in \B$ are sufficiently small that $\ol{U}_{\beta_1}\cap \ol{U}_{\beta_2}=\emptyset$ for $\beta_1\ne \beta_2$. Let $\ti{\sigma}_\beta$ be the restriction of $\ti{\sigma}$ to $T^*U_\beta$; since $\ti{\sigma}_\beta \upharpoonright T_G^*U_\beta$ is invertible outside the compact subset $\C_\beta \subset T_G^*U_\beta$, we obtain a well-defined class $[\ti{\sigma}_\beta] \in \tn{K}^0_G(T^*_GU_\beta)$. Then the excision property (and \eqref{e:extendind}) says that the equation 
\begin{equation} 
\label{e:locformula}
\index_G(\dirac^L)=\index_G([\ti{\sigma}])=\sum_{\beta \in \B}\index_G([\ti{\sigma}_\beta]) 
\end{equation}
holds in $R^{-\infty}(G)$.  Equation \eqref{e:locformula} (or more explicit versions thereof, cf. \cite[Theorem 8.6]{WittenNonAbelian}) are referred to as \emph{non-abelian localization formulas} in K-theory.

\subsection{The $[Q,R]=0$ theorem.}\label{s:qr}
We now formulate a $[Q,R]=0$ theorem for compact Hamiltonian log symplectic manifolds in the special case that $G$ acts freely on $\mu^{-1}(0)$.

Let $(M,\Z,\omega,\mu)$ be a Hamiltonian log symplectic manifold. Assume $G$ acts freely on $\mu^{-1}(0)$.  Then one can show \cite{LLSS1} that 
\begin{enumerate}
\item $\mu^{-1}(0)$ and $M_0=\mu^{-1}(0)/G$ are smooth. An orientation on $M$ induces an orientation on $M_0$.
\item For each $Z \in \Z$ that intersects $\mu^{-1}(0)$, $Z_0=(Z\cap \mu^{-1}(0))/G$ is a smooth hypersurface in $M_0$. Hence there is a normal crossing divisor $\Z_0$ in $M_0$. If $\Z$ admits global defining functions then $\Z_0$ does as well.
\item There is a unique log symplectic form $\omega_0$ on $(M_0,\Z_0)$ such that $p^*\omega_0=\iota^*\omega$, where $p\colon \mu^{-1}(0)\rightarrow M_0$ is the quotient map and $\iota\colon \mu^{-1}(0)\hookrightarrow M$ is the inclusion map.  If $L$ is a $G$-equivariant prequantum line bundle on $M$, then $L_0=L|_{\mu^{-1}(0)}/G$ is a prequantum line bundle on $M_0$.
\item There is a $G$-equivariant diffeomorphism $\varphi \colon \mu^{-1}(0) \times U_{\g^*} \rightarrow U\subset M$, where $U_{\g^*}$ is a $G$-invariant open neighborhood of $0 \in \g^*$ and $U$ is a $G$-invariant open neighborhood of $\mu^{-1}(0)$ in $M$ having the two properties: (i) for each $Z \in \Z$, $\varphi^{-1}(Z\cap U)=(Z\cap \mu^{-1}(0))\times U_{\g^*}$, (ii) if $\tn{pr}_j$, $j=1,2$ are the projection maps to the first and second factors respectively in $\mu^{-1}(0)\times U_{\g^*}$, then
\[ \varphi^*\omega=\tn{pr}_1^*p^*\omega_0+\d \pair{\tn{pr}_2}{\tn{pr}_1^*\theta} \]
where $\theta$ is a connection on the principal $G$-bundle $p\colon \mu^{-1}(0)\rightarrow M_0$.
\end{enumerate}
All except (d) are routine.

\begin{theorem}
\label{t:qr}
Let $(M,\Z,\omega,\mu,L)$ be a compact prequantized Hamiltonian log symplectic manifold such that $\Z$ admits global defining functions and $\mu$ is proper.  Assume $G$ acts freely on $\mu^{-1}(0)$ and let $(M_0,\Z_0,\omega_0)$ be the reduced log symplectic manifold with prequantum line bundle $L_0$, and the induced orientation.  Then
\[ \index_G(\dirac^L_M)^G=\index(\dirac_{M_0}^{L_0}).\]
\end{theorem}

\begin{remark}
In \cite{BLSbsympl} this result was proved in the special case where the divisor $\Z$ has no crossings and the modular weights are non-zero. Under these assumptions the complement $M \backslash \cup \Z$ is a symplectic Hamiltonian $G$-space with proper momentum map.  This made it possible to deduce the result from a theorem of Ma and Zhang. 
\end{remark}

\begin{remark}
Theorem \ref{t:qr} is closely analogous to the symplectic $[Q,R]=0$ theorem \cite{MeinrenkenSymplecticSurgery}. We comment that there is a very general $[Q,R]=0$-type theorem due to Paradan and Vergne \cite{ParadanVergneSpinc} that applies to arbitrary $G$-equivariant Spin$_c$ Dirac operators on compact manifolds, hence applies to the operator $\dirac^L_M$. However the Paradan-Vergne result is \emph{different} from Theorem \ref{t:qr}; in particular their result employs a different momentum map and moreover $\index_G(\dirac^L_M)^G$ can have contributions from more than one of its level sets. 
\end{remark}

Using a shifting trick established in \cite{LLSS1}, Theorem \ref{t:qr} implies a similar result for the multiplicity of any irreducible representation $V_\lambda \in R(G)$ with highest weight $\lambda$:
\begin{corollary}
\label{c:shifting}
Let $(M,\Z,\omega,\mu)$ be as in Theorem \ref{t:qr}.  Assume $G_\lambda$ acts freely on $\mu^{-1}(\lambda)$.  Let $M_{\lambda}=\mu^{-1}(\lambda)/G_\lambda$ be the reduced log symplectic manifold \cite{LLSS1}, with prequantum line bundle $L_{\lambda}=(L|_{\mu^{-1}(\lambda)}\otimes \bC_{-\lambda})/G_\lambda$, and the induced orientation.  Then
\[ \big(\index_G(\dirac^L_M)\otimes V_\lambda^*\big)^G=\index(\dirac_{M_\lambda}^{L_\lambda}).\]
\end{corollary}
\begin{proof}
Let $\O\subset \g^*$ denote the coadjoint orbit containing $-\lambda$ equipped with the standard Kirillov-Kostant-Souriau symplectic form and compatible complex structure.  It admits the prequantum line bundle $E=G\times_{G_\lambda}\bC_{-\lambda}\rightarrow G/G_\lambda\simeq \O$, and the index of the corresponding twisted Dirac operator is $V_\lambda^*$.  By the shifting trick $M_{\lambda}$ is the reduced space at $0$ of the product $M\times \O$.  Applying Theorem \ref{t:qr} to $M\times\O$ yields
\[ \big(\index_G(\dirac^L_M)\otimes V_\lambda^*\big)^G=\index_G(\dirac_{M\times\O}^{L\boxtimes E})^G=\index(\dirac_{M_\lambda}^{L_\lambda}).\]
\end{proof}

We will deduce Theorem \ref{t:qr} from general theorems due to Paradan \cite{ParadanRiemannRoch} and Paradan-Vergne \cite{WittenNonAbelian}; these authors studied equation \eqref{e:locformula} in detail for deformations of Dirac-type symbols via $G$-equivariant maps $M\rightarrow \g$, as is the case in our situation, where the deformation $\ti{\kappa}$ is associated to the $G$-equivariant map $\ti{\mu}\colon M \rightarrow \g^*\simeq \g$.  We first prove a lemma.

\begin{lemma}
Let $(M,\Z,\omega,\mu)$ and $L$ be as in Theorem \ref{t:qr}. Then taking $G$-invariants of both sides of the non-abelian localization formula \eqref{e:locformula} yields
\begin{equation} 
\label{e:invt0}
\index_G(\dirac^L)^G=\index_G([\sigma_0])^G,
\end{equation}
or in other words, the only contribution to the trivial representation in \eqref{e:locformula} comes from $\beta=0$.  
\end{lemma}
\begin{proof}
A criterion for exactly this result is given in \cite[Theorem 9.6]{WittenNonAbelian} (see also Definitions 7.1, 9.2 as well as the paragraph below Theorem 8.6 of \cite{WittenNonAbelian} for explanations of the notation). To state it, we introduce some notation. We shall need to consider vector bundles carrying more than one complex structure below.  If $V$ is a real vector bundle equipped with a complex structure $I \in \Gamma(\End(V))$, then we write $(V,I)^{1,0}$, resp. $(V,I)^{0,1}$ for the $+\sqrt{-1}$, resp. $-\sqrt{-1}$ eigenbundles of $I$ in $V\otimes \bC$. 

Let $0 \ne \beta \in \B$, and let $M_\beta \subset M^\beta$ be the union of components of $M^\beta$ that intersect $\mu^{-1}(\beta)$ non-trivially. Let $\nu$ be the normal bundle to $M_\beta$ in $M$.  We may equip $\nu$ with the complex structure $J_\beta=A_\beta/|A_\beta|$, where $A_\beta$ denotes the (skew-symmetric) endomorphism of $\nu$ given by the action of $\beta$.  Then the complex exterior algebra $\wedge (\nu,J_\beta)^{0,1}$ is a $\Cl(\nu)$-module.  Let $\g_\beta \subset \g$ be the infinitesimal stabilizer of $\beta$ under the adjoint action.  We regard $\g/\g_\beta$ as a complex vector space, with complex structure $j_\beta=\ad_\beta/|\ad_\beta|$; equivalently, since $\beta \in \t_+$, $(\g/\g_\beta,j_\beta)^{1,0}$ is a sum of positive root spaces. Then the criterion in \cite[Theorem 9.6]{WittenNonAbelian} states that it suffices to show that the (locally constant, real) eigenvalue for the action of $b=-\sqrt{-1}\beta$ on the $\Cl(TM_\beta)$-spinor module 
\[ S_\beta=\Hom_{\Cl(\nu)}(\wedge (\nu,J_\beta)^{0,1},S\otimes L|_{M_\beta})\otimes \tn{det}(\g/\g_\beta,j_\beta)^{0,1} \]
is \emph{positive}. (Note that there is indeed only a single eigenvalue for each component of $M_\beta$ by Schur's lemma, because $b$ commutes with the $\Cl(TM_\beta)$-action.)  This eigenvalue is equal to one half the action of $b$ on the anti-canonical line bundle of $S_\beta$:
\begin{equation} 
\label{e:anticanon}
\L_\beta=\tn{det}(\nu,J_\beta)^{1,0}\otimes \big(\tn{det}(\g/\g_\beta,j_\beta)^{0,1}\big)^2\otimes \tn{det}(T_\Z M,J)^{1,0}|_{M_\beta}\otimes L^2|_{M_\beta}
\end{equation}
where we used equation \eqref{e:detline} for the anti-canonical line bundle of $S$.  The hypersurfaces in $\Z$ are automatically transverse to $M_\beta$, because they are preserved by the $G$-action. Thus $T_\Z M|_{M_\beta}\simeq T_{\Z \cap M_\beta}M_\beta\oplus \nu$, and since $\beta$ acts trivially on $T_{\Z\cap M_\beta}M_\beta$, there is no harm in omitting this in \eqref{e:anticanon}, i.e. it suffices to consider the action of $b$ on the line bundle
\begin{equation} 
\label{e:anticanon2}
\L_\beta'=\tn{det}(\nu,J_\beta)^{1,0}\otimes \big(\tn{det}(\g/\g_\beta,j_\beta)^{0,1}\big)^2\otimes \tn{det}(\nu,J)^{1,0}|_{M_\beta}\otimes L^2|_{M_\beta}
\end{equation}
We are now in essentially the same situation encountered in the symplectic setting.  Since the action is locally constant on $M_\beta$, it suffices to study \eqref{e:anticanon2} at a single point $m$ in each component of $\C_\beta$. The eigenvalues for the action of $b$ on $(\g/\g_\beta,j_\beta)^{0,1}$ are negative by construction.  The fibre $\nu_m$ is a direct sum $\nu'_m\oplus \g/\g_\beta$, where $\g/\g_\beta$ is identified with the orthogonal complement to $\g_\beta$ in $\g$, and is embedded in $\nu_m$ as a subset of the $G$-orbit directions.  Note that the complex structure $j_\beta$ on $\g/\g_\beta$ is compatible with the restriction of the log symplectic form to $\g/\g_\beta \subset \nu_m$: for $X \in \g/\g_\beta$,
\[ \omega_m(X_M,(\ad_\beta X)_M)=-\d_m\pair{\mu}{X}\big((\ad_\beta X)_M\big)=-\pair{\ad_\beta^2 X}{X}=\|\ad_\beta X\|^2>0\] 
where in the first equality we used the momentum map equation, in the second equality we used equivariance of the momentum map ($\d_m \mu(Y_M)=-\ad_Y^*\beta=\ad_\beta Y$, as $\mu(m)=\beta$ and using $\g\simeq \g^*$), and in the third equality we used that $\ad_\beta$ is skew adjoint.  Thus performing a small homotopy of $J$ if necessary, we may assume that $J$ preserves $\g/\g_\beta\subset \nu_m$ and equals $j_\beta$ on this subspace; it is then clear that the eigenvalues of $b$ on $(\g/\g_\beta,J|_{\g/\g_\beta})^{1,0}$ are positive. Hence after cancelling the $\g/\g_\beta$ contributions, we are left to consider the action of $b$ on the complex line
\[ \tn{det}(\nu_m',J_\beta)^{1,0}\otimes \tn{det}(\nu_m',J)^{1,0}\otimes L_m^2.\]
The eigenvalues of $b$ on $(\nu_m',J_\beta)^{1,0}$ are positive by construction, the eigenvalues on $(\nu_m',J)^{1,0}$ have mixed signs, but the negative ones cancel with the corresponding eigenvalues for the action on $(\nu_m',J_\beta)^{1,0}$.  Hence the eigenvalue for the action of $b$ on $\tn{det}(\nu_m',J_\beta)^{1,0}\otimes \tn{det}(\nu_m',J)^{1,0}$ is non-negative.  On the other hand, equation \eqref{e:Kostant} shows that the eigenvalue of $b$ on $L_m$ is $2\pi \|\beta\|^2>0$. 
\end{proof}

\begin{proof}[Proof of Theorem \ref{t:qr}]
Since $G$ acts freely on $\mu^{-1}(0)$, there is a direct sum decomposition
\[ TM|_{\mu^{-1}(0)}\simeq p^*TM_0\oplus \underline{\g} \oplus \underline{\g}^*.\]
The trivial bundle $\underline{\g}\oplus \underline{\g}^*\simeq \underline{\g}_\bC$ carries a canonical complex structure, and hence a spinor module $\wedge \underline{\g}_\bC$. We therefore obtain an induced spinor module for $TM_0$:
\begin{equation} 
\label{e:spinorquotient}
S_0'=\tn{Hom}_{\Cl(\g_\bC)}(\wedge \ul{\g}_\bC,S|_{\mu^{-1}(0)})/G.
\end{equation}
By \cite[Theorem 9.6]{WittenNonAbelian}, $\index_G(\sigma_0)^G$ equals the index of a Dirac operator on $M_0$, acting on sections of $S_0'\otimes L_0$. To complete the proof it suffices to show that $S_0'$ is homotopic to the spinor module $S_0$ on $M_r$ determined by its log symplectic structure and orientation.  

We can deduce this from the normal form given in item (d) near the beginning of Section \ref{s:qr}.  The map $p\circ \tn{pr}_1 \colon U \rightarrow M_0$ is a fibre bundle with fibres diffeomorphic to $T^*G$, and by item (d), $T_{\Z\cap U} U\simeq (p\circ \tn{pr}_1)^* T_{\Z_0}M_0\oplus \underline{\g}_\bC$ (identifying $\g\oplus \g^*$ with $\g_\bC$ also).  Using the formula for the log symplectic form in item (d), it follows that the complex structure $(p\circ \tn{pr}_1)^*J_0\oplus J_{\g_\bC}$ on $T_{\Z\cap U}U$ is compatible with $\omega$, where $J_0$ is a $\omega_0$-compatible complex structure on $T_{\Z_0}M_0$ and $J_{\g_\bC}$ is the standard complex structure on $\ul{\g}_\bC$.  Using Remark \ref{r:fibre}, the stable isomorphism may be chosen such that the spinor module $S|_U=(p\circ \tn{pr}_1)^\ast S_0 \hat{\otimes} \wedge \ul{\g}_\bC$.  This implies the desired result.
\end{proof}

\begin{example}
\label{ex:2sphere5}
We return to Example \ref{ex:2sphere4}, carrying over the notation introduced there. Suppose $n_1\le n_2$ (the other case being analogous).  The reduced spaces $\mu^{-1}(j)/S^1$, $j \in \bZ$ are either (i) empty if $j < n_1$, (ii) a single point with positive orientation if $n_1\le j<n_2$, (iii) a pair of points with opposite orientations if $n_2\le j$.  Cases (i), (iii) have vanishing quantization. Thus according to the $[Q,R]=0$ theorem
\[ \index_{S^1}(\dirac^L)(t)=\sum_{j=n_1}^{n_2-1} t^j \]
in agreement with \eqref{e:AB}.
\end{example}

\subsection{Brief remarks on the singular case.}
Let $(M,\Z,\omega,\mu)$ be a compact log symplectic Hamiltonian $G$-space.  If the action of $G$ on $\mu^{-1}(0)$ is only locally free instead of free, then the reduced space $M_0=\mu^{-1}(0)/G$ is an orbifold. In order to generalize Theorem \ref{t:qr} to this situation, one option is to define the (orbifold) spinor module on the reduced space $M_0$ by equation \eqref{e:spinorquotient}, and define the quantization of $M_0$ to be the index of the corresponding Dirac operator twisted by $L_0$.  Then Theorem \ref{t:qr} holds with the same proof. Alternatively one can extend the definitions and constructions to orbifolds.  We briefly indicate how this can be done.

Let $M^n$ be an orbifold. For $p \in M$ let $\Gamma_p$ denote the isotropy group; we allow the isotropy group to act non-effectively in orbifold charts.  We define a (simple) normal crossing divisor in $M$ to be a finite collection $\Z$ of codimension $1$ suborbifolds such that if $Z_1,...,Z_k \in \Z$ and $p \in Z_1 \cap \cdots \cap Z_k$, there is a local orbifold chart $\varphi \colon U \rightarrow \bR^n/\Gamma_p$ centred at $p$ such that
\begin{enumerate}
\item $\bR^n=\bR^k\times \bR^{n-k}$ and $\Gamma_p$ acts only on the second factor;
\item for $j=1,...,k$, $\varphi$ maps $Z_j \cap U$ into a subset of $\{x_j=0\}/\Gamma_p$, where $x_1,...,x_k$ are the coordinates in $\bR^k$.
\end{enumerate}
Assuming $\Z$ admits global defining functions, this definition implies that each intersection $Z_1\cap \cdots \cap Z_k$ has a neighborhood in $M$ of the form $\bR^k\times (Z_1 \cap \cdots \cap Z_k)$, a product of a manifold with the suborbifold $Z_1 \cap \cdots \cap Z_k$.

By working in orbifold charts, one verifies that the sheaf of smooth sections of $TM$ tangent to all $Z \in \Z$ forms the sheaf of smooth sections of an orbifold vector bundle $T_\Z M$, which can moreover be equipped with the obvious analogue of the Lie algebroid structure present in the manifold case. The definition of a log symplectic form then carries over. Using the remark above regarding existence of a neighborhood of $Z_1\cap \cdots \cap Z_k$ of product form, one can carry through the construction of the stable isomorphism between $TM$, $T_\Z M$. Hence $TM$ becomes stably almost complex as in the manifold case, and the rest of the definition of the quantization carries through. The statement and proof of Theorem \ref{t:qr} then generalize to the case where $G$ acts locally freely on $\mu^{-1}(0)$.

In the still more general setting with no assumptions on the action of $G$ on $\mu^{-1}(0)$, one can obtain an analogue of Theorem \ref{t:qr} using a shift desingularization to a nearby weakly regular value as in \cite{MeinrenkenSjamaar}.

\section{Toric log symplectic manifolds.}\label{s:toric}
In this section we specialize our result to the toric log symplectic manifolds defined and classified in \cite{guillemin2014toric} (for the case without crossings) and \cite{gualtieri2017tropical} (for the case with crossings). We begin by giving a brief introduction to the framework in \cite{gualtieri2017tropical}. 

Recall that in the symplectic case, there is a systematic procedure for constructing all compact toric examples. In brief, given a convex polytope $\Delta \subset \t^*$ satisfying the Delzant condition, one obtains a toric symplectic manifold from the trivial principal $T$-bundle $\Delta \times T$ by performing symplectic cutting along each of the codimension $1$ faces of $\Delta$. The authors of \cite{gualtieri2017tropical} generalized this construction to the log symplectic case, and proved a corresponding classification theorem. Two of the main new features they discovered were (i) the codomain $\t^*$ of the momentum map can be replaced by a more complicated `tropical welded space'; (ii) the principal $T$-bundle used in the construction can be non-trivial.

\subsection{Tropical welded spaces and log affine polytopes.}
Let $n>0$ be even and let $\t^*$ be a real vector space of dimension $n/2$ (soon to be the dual of the Lie algebra of a compact torus $T$). To avoid confusion below, we will use the notation $\T^*$ to denote the space $\t^*$ viewed as an abelian Lie group with Lie algebra $\t^*$. A \emph{tropical welded space} \cite[Definition 3.8]{gualtieri2017tropical} $(\Sigma,\D,\xi)$ is a smooth connected $\T^*$-manifold $\Sigma^{n/2}$, equipped with a normal crossing divisor $\D$ and a closed non-degenerate $\t^*$-valued log $1$-form $\xi \in \Omega^1(\Sigma,\D)\otimes \t^*$, such that the image of $[\xi]$ in $H^1(\Sigma)\otimes \t^*$ under the Mazzeo-Melrose map vanishes.  Here `non-degenerate' means that the map $T_\D \Sigma \rightarrow \Sigma\times \t^*$ induced by $\xi$ is an isomorphism; the closedness of $\xi$ implies that this map is in fact an isomorphism of Lie algebroids, where $\Sigma \times \t^*$ is identified with the action Lie algebroid $\Sigma\rtimes \t^*$ for the $\T^*$ action.  

Any such space may be constructed by gluing together a number of copies of partial compactifications of $\t^*$ (\cite[Section 3.2]{gualtieri2017tropical}), and in particular the connected components of $\Sigma \backslash \cup \D$ are affine spaces modelled on $\t^*$. Let $\Sigma_0 \subset \Sigma \backslash \cup \D$ be a connected component, and choose an origin $p_0 \in \Sigma_0$ so that $\Sigma_0$ becomes identified with $\t^*$. Given any point $p \in \Sigma \backslash \cup \D$, we may obtain an element of $\t^*$ by taking the principal value integral of $\xi$ along any smooth curve from $p_0$ to $p$ which is transverse to $\D$; the vanishing of the component of $[\xi]$ in $H^1(\Sigma)\otimes \t^*$ implies that the result does not depend on the choice of curve.  Thus by choosing a single basepoint in $\Sigma \backslash \cup \D$ one obtains a consistent identification of each component of $\Sigma\backslash \cup \D$ with $\t^*$.

A \emph{log affine hyperplane} in $\Sigma$ is a smooth connected embedded hypersurface $H\subset \Sigma$ not contained in $\D$, which is preserved under translations by some hyperplane in $\T^*$.  An \emph{admissible log affine polytope} (compare \cite[Definition 5.3]{gualtieri2017tropical}) is a compact submanifold with corners $\Delta^n \subset \Sigma^n$, such that the closures of the codimension $1$ strata of $\partial \Delta$ are contained in log affine hyperplanes, and every stratum of $\partial \Delta$ intersects every stratum of $\D$ transversely.  An oriented admissible log affine polytope has a well-defined regularized volume, defined as the principal value integral of the top power of $\xi$ over $\Delta$. One says $\Delta$ is \emph{convex} if its intersection with each component of $\Sigma \backslash \cup \D$ is convex. There is also a version of the Delzant condition, see \cite[Definition 5.20]{gualtieri2017tropical}.

\subsection{Construction of toric log symplectic manifolds.}
Let $T$ be a compact torus with Lie algebra $\t$.  Let $(\Sigma,\D,\xi)$ be a tropical welded space where $\xi \in \Omega^1(\Sigma,\D)\otimes \t^*$.  Let $\pi_P\colon P\rightarrow \Sigma$ be a principal $T$-bundle. There is a canonically defined \emph{obstruction class} $\Tr(c_1(P)\wedge \xi) \in H^3(\Sigma,\D)$ \cite[Definition 4.15]{gualtieri2017tropical}, and in \cite[Theorem 4.16]{gualtieri2017tropical} it is shown that if the obstruction class vanishes then the total space of $P$ admits log symplectic forms $\omega_P$ with poles along the divisor $\Z_P=\pi^{-1}(\D)$ satisfying
\[ \iota(X_P)\omega_P=-\pi_P^*\pair{\xi}{X}.\]
The space of equivalence classes of such log symplectic forms up to $T$-equivariant symplectomorphisms inducing the identity on $\Sigma$ is an affine space of the real vector space $H^2(\Sigma,\D)$ \cite[Theorem 4.16]{gualtieri2017tropical}.

If $\Delta \subset \Sigma$ is an admissible convex log affine polytope satisfying the Delzant condition, then one can perform a log symplectic version of symplectic cutting on $\pi^{-1}(\Delta)\subset P$ along the faces of $\Delta$ in order to obtain a smooth closed log symplectic $T$-manifold $(M^{2n},\Z,\omega)$, with induced map $\pi\colon M\rightarrow \Delta \subset \Sigma$ satisfying
\[ \iota(X_M)\omega=-\pi^*\pair{\xi}{X}.\]
We refer the reader to \cite[Theorem 5.18, Corollary 5.21, Theorem 6.3]{gualtieri2017tropical} for details and further results.

The modular weights of $(M,\Z,\omega)$ are minus the residues of $\xi$ at the various hypersurfaces of $\D$. Choosing a basepoint $p_0 \in \Sigma\backslash \cup \D$, each connected component of $\Sigma \backslash \cup \D$ becomes identified with $\t^*$, and hence the map $\pi$ determines a map
\begin{equation} 
\label{e:mmSigma}
\mu \colon M\backslash \cup \Z \rightarrow \t^* 
\end{equation}
which is a momentum map in the sense of Definition \ref{d:Hamiltonian}. It is explained in \cite[Proposition A.5]{gualtieri2017tropical} that the residues $c_1,...,c_k$ associated to a collection $D_1,...,D_k \in \D$ such that $D_1 \cap \cdots \cap D_k \ne \emptyset$ are linearly independent.  In particular it follows from Lemma \ref{l:proper} that $\mu$ is proper.

\subsection{Quantization of toric log symplectic manifolds.}
By the transversality assumption on the strata of $\Delta$, the vertices of $\Delta$, which are the images of the $T$-fixed point submanifolds of $M$, lie in $\Sigma \backslash \cup \D$. We may choose a basepoint in $\Sigma \backslash \cup \D$ to arrange that one of these vertices lies in the weight lattice $\Lambda$ for $T$ once the components of $\Sigma \backslash \cup \D$ are identified with $\t^*$ (for example, choose the basepoint to be one of the vertices of $\Delta$). As a corollary of Remark \ref{r:prequantizable} we have:
\begin{corollary}
A toric log symplectic manifold $(M,\Z,\omega,\mu)$ is prequantizable (in the sense of Definition \ref{d:prequantizable}) if and only if the image of $[\omega]$ in $H^2(M)$ is integral.
\end{corollary}

Assume the image of $[\omega] \in H^2(M)$ is integral, and let $(L,\nabla^L)$ be prequantum data.  If $\Delta$ is oriented and $\D$ admits global defining functions, then the quantization $\index_T(\dirac^L) \in R(T)$ is defined (Definition \ref{d:RRnumber}), and depends only on the isomorphism class of the $T$-equivariant line bundle $L$ as well as the chosen orientation of $\Delta$ (reversing the orientation reverses the overall sign).  Since $\mu$ is (automatically) proper, Corollary \ref{c:shifting} of the $[Q,R]=0$ theorem applies, and yields the following description of the quantization.  Let $\Sigma_1,...,\Sigma_k$ be the connected components of $\Sigma \backslash \cup \D$. For each $1\le j \le k$, let $o_j \in \{\pm 1\}$ be the parity of the number hypersurfaces of $\D$ crossed by a smooth curve connecting the basepoint to a point of $\Sigma_j$.  Under the identification of $\Sigma_j$ with $\t^*$, $\Sigma_j\cap \Delta$ becomes a (possibly non-compact) polyhedron $\Delta_j \subset \t^*$ with finitely many edges.  Let $[\Delta_j \cap \Lambda]$ denote the characteristic function of $\Delta_j \cap \Lambda$.
\begin{corollary}
\label{c:tropicalquant}
The multiplicity of the representation $\bC_\lambda$, $\lambda \in \Lambda$ in $\index_T(\dirac^L)$ is
\[ \sum_{j=1}^k o_j[\Delta_j\cap \Lambda](\lambda).\]
\end{corollary}
\noindent It is a consequence of Corollary \ref{c:shifting} that the RHS vanishes for all but finitely many $\lambda \in \Lambda$. Corollary \ref{c:tropicalquant} is a generalization of the following well-known fact from the compact toric symplectic case: the quantizing Hilbert space is a finite dimensional representation of $T$, and the set of weights that occurs is $\Delta \cap \Lambda$, each with multiplicity $1$.

\begin{remark}
The map $\pi \colon M \rightarrow \Delta \subset \Sigma$ allows us to define a more refined `quantization', as follows.  The Dirac operator determines a class $[\dirac^L]$ in the K-homology group $K_0^T(M)$, and we may take its push-forward $\pi_\ast[\dirac^L]\in K_0^T(\Delta)\simeq K_0(\Delta)\otimes R(T)$. Since $\Delta$ may have non-trivial topology, this K-homology class may contain additional information. Pushing forward further under the map $\Delta \rightarrow \pt$ recovers $\index_T(\dirac^L) \in R(T)$.
\end{remark}

\bibliographystyle{amsplain}
\providecommand{\bysame}{\leavevmode\hbox to3em{\hrulefill}\thinspace}
\providecommand{\MR}{\relax\ifhmode\unskip\space\fi MR }
\providecommand{\MRhref}[2]{%
  \href{http://www.ams.org/mathscinet-getitem?mr=#1}{#2}
}
\providecommand{\href}[2]{#2}

\end{document}